\documentclass[draft, 11 pt]{amsart}

\usepackage{amssymb,amsmath,amsthm,amscd,mathrsfs,graphicx}
\usepackage[matrix,arrow,tips,curve]{xy}
\usepackage[english]{babel}
\usepackage[leqno]{amsmath}
\usepackage{amssymb,amsthm}
\usepackage{amscd}
\usepackage{enumerate}
\usepackage{layout}
\usepackage{fullpage}
\usepackage{xcolor}

\newtheorem{thm}{Theorem}[section]
\newtheorem{theoremalpha}{Theorem}

\newtheorem{lemma}[thm]{Lemma}

\newtheorem{cor}[thm]{Corollary}

\newtheorem{fact}[thm]{Fact}
\theoremstyle{definition}
\newtheorem{remark}[thm]{Remark}
\newtheorem{defi}[thm]{Definition}
\newtheorem{convention}[thm]{}

\numberwithin{equation}{section}

\newcommand{\isoto}{\stackrel{\backsim}{\longrightarrow}}

\newcommand{\w}{\widetilde}
\newcommand{\ov}{\overline}

\newcommand{\wh}{\widehat}

\newcommand{\Aut}{\operatorname{Aut}}
\newcommand{\Hom}{\operatorname{Hom}}
\newcommand{\End}{\operatorname{End}}

\newcommand{\Spec}{\operatorname{Spec}}

\DeclareMathOperator{\id}{id}

\newcommand{\SL}{\operatorname{SL}}
\newcommand{\Stab}{\operatorname{Stab}}
\newcommand{\Gr}{\operatorname{Gr}}

\newcommand{\calI}{{ \mathcal I}}

\newcommand{\calO}{{ \mathcal O}}

\newcommand{\calQ}{{ \mathcal Q}}

\newcommand{\calX}{{\mathcal X}}

\newcommand{\bbC}{{\mathbb C}}

\newcommand{\bbG}{{\mathbb G}}

\newcommand{\bbN}{{\mathbb N}}

\newcommand{\bbP}{{\mathbb P}}
\newcommand{\bbQ}{{\mathbb Q}}

\newcommand{\bbZ}{{\mathbb Z}}

\newcommand{\Mgbar}{\ov{\operatorname{M}}_g}

\newcommand{\pdbst}{\bar{J}_{d,g}}
\newcommand{\pdbstFunc}{\bar{J}_{d,g}^{\sharp}}

\newcommand{\Quot}{\operatorname{Quot}}

\newcommand{\Def}{\operatorname{Def}}
\newcommand{\Spf}{\operatorname{Spf}}

\newcommand{\art}{ \operatorname{Art}_{k} }
\newcommand{\h}[1][R]{\operatorname{Spf}(#1)}

\newcommand{\Sets}{\text{Sets}}
\newcommand{\sets}{\text{Sets}}

\newcommand{\el}[1][e]{\overset{\leftarrow}{#1}}
\newcommand{\er}[1][e]{\overset{\rightarrow}{#1}}

\newcommand{\node}{\calO_{0}}
\newcommand{\nodenm}{\tilde{\calO}_{0}}

\newenvironment{sis}{\left\{\begin{aligned}}{\end{aligned}\right.}

\makeatletter
\@namedef{subjclassname@2010}{ \textup{2010} Mathematics Subject
Classification} \makeatother

\begin{document}

\title[Deformation Theory]{The Local Structure of Compactified Jacobians}

\author[Casalaina-Martin]{Sebastian Casalaina-Martin}
\address{University of Colorado at Boulder, Department of Mathematics,
Colorado (USA)}
\email{casa@math.colorado.edu}

\author[Kass]{Jesse Leo Kass}
\address{University of South Carolina, Department of Mathematics, 
South Carolina (USA)}
\email{kassj@math.sc.edu}

\author[Viviani]{Filippo Viviani}
\address{Roma Tre University, Department of Mathematics and Physics, 
Rome (Italy)}
\email{viviani@mat.uniroma3.it}

\thanks{The first author was supported by NSF grant DMS-1101333.  The second author was supported by NSF grant DMS-0502170. The third author is supported
 by the MIUR project \textit{Spazi di moduli e applicazioni} (FIRB 2012),  by CMUC and by the FCT-grants   PTDC/MAT-GEO/0675/2012 and EXPL/MAT-GEO/1168/2013.}


\begin{abstract}
This paper studies the local geometry of compactified Jacobians.
The main result is a presentation of the completed local ring of the compactified Jacobian of a nodal curve as an explicit ring of invariants described in terms of the dual graph of the
curve.  The authors have investigated the geometric and combinatorial properties of these rings in previous work, and consequences for compactified Jacobians are presented in this
paper.
Similar results are given for the local structure of the universal compactified Jacobian over the moduli space of stable curves.
\end{abstract}

\maketitle

\bibliographystyle{amsplain}

\vspace{-0.5cm}

\section{Introduction} \label{Sec: Intro}
This paper studies the local geometry of compactified Jacobians associated to  nodal curves.  These are projective varieties that play a role similar to that of
the Jacobian variety for a non-singular curve.   Recall that a Jacobian can be viewed as the moduli space of   line bundles (of fixed degree) on a non-singular curve.
A compactified Jacobian is an analogous parameter space associated to a nodal curve.
A major barrier to  constructing these spaces is that, while the moduli space of fixed degree line bundles on a nodal curve exists,
it typically does not have nice properties: often it has infinitely many connected components (i.e.~is not of finite type), and these components fail to be proper.  To construct a well-behaved compactified Jacobian,
one must modify the moduli problem.
There are a number of different ways to do this, and the literature on the compactification
problem is vast (e.g.~\cite{ishida},  \cite{dsouza}, \cite{OS},  \cite{altman80}, \cite{Ses}, \cite{caporaso}, \cite{simpson}, \cite{Pan}, \cite{jarvis}, \cite{esteves01}, \cite{melo07}, \cite{ftt}).

 Geometric Invariant Theory (GIT) provides a general framework for these types of compactification problems, and in this approach, a compactified Jacobian $\bar{J}(X)$ of a nodal
curve $X$
is constructed as a coarse moduli space of certain line bundles together with their degenerations: rank $1$, torsion-free sheaves.
The sheaves parameterized by $\bar{J}(X)=\bar{J}_{\phi}(X)$ are those rank $1$, torsion-free sheaves that are semistable with respect to a numerical polarization $\phi$ (see Definitions \ref{D:numpol}
and \ref{D:phi-ss}). As explained in \cite[Sec. 2]{MV},  this semistability condition generalizes the other know semistability conditions  that appear in the work of Oda-Seshadri \cite{OS}, Seshadri \cite{Ses}, Esteves \cite{esteves01} as well as the more familiar slope semistability condition with respect to an ample line bundle that appears in the work of Simpson \cite{simpson}.
In general, compactified Jacobians are non-fine moduli spaces because typically there are non isomorphic semistable sheaves $I$ and $I'$ that correspond to the same point  $[I]=[I']\in \bar{J}(X)$. 
Indeed, this happens precisely when two Jordan-H\"older filtrations of $I$ and $I'$ have the same stable factors. If a Jordan-H\"older filtration of $I$ splits, i.e. if $I$ is  the direct sum of stable sheaves supported on subcurves of $X$, we say that $I$ is polystable. Therefore, given a point $x\in \bar{J}(X)$, there exists a unique polystable sheaf $I$ such that $[I]=x\in \bar{J}(X)$; see \S  \ref{SS:compJac} for more details.

One motivation for constructing compactified Jacobians is that they provide degenerations of
Jacobian varieties.  Given a family of non-singular curves specializing to a
nodal curve, the compactified Jacobian of the nodal fiber fits into a family that extends the
family consisting of the Jacobians of the non-singular fibers.   Note that because the coarse moduli
space of stable curves $\Mgbar$ does not admit a universal curve, this does not  imply that the compactified Jacobians
 fit into a family over $\Mgbar$.  However,
Caporaso \cite{caporaso} (and later Pandharipande \cite{Pan}) has constructed a family $\Phi: \pdbst \to \Mgbar$ (which we call the universal compactified Jacobian) of projective
schemes that extends the Jacobian of the generic genus $g$ curve.  The fiber of $\Phi$ over a stable curve $X$ is isomorphic to a certain compactified Jacobian of $X$, modulo its automorphism group
(see Fact~\ref{Fact: Comparison}). 

The main result of this paper describes the local geometry of both a compactified Jacobian $\bar{J}(X)$ of a nodal curve $X$, and  of  the universal compactified Jacobian $\pdbst$ at a point corresponding
to   an automorphism-free stable curve.

\begin{theoremalpha} \label{Thm: MainThmA}
 Let $X$ be a nodal curve of arithmetic genus $g(X)$,  let $I$ be a rank $1$, torsion-free sheaf on $X$, and let $\Sigma$ be the set of nodes
where $I$ fails to be locally free.
Set $\Gamma=\Gamma_X(\Sigma)$ to be the dual graph of any curve obtained from $X$ by \emph{smoothing} the nodes \emph{not} in $\Sigma$.  Fix an arbitrary orientation on $
\Gamma$, and
denote by $V(\Gamma)$, $E(\Gamma)$, and  $s,t:E(\Gamma)\to V(\Gamma)$ the vertices, edges and source and target maps respectively.  Set $b_1(\Gamma)=\#E(\Gamma)-
\#V(\Gamma)+1$.
Let
$$
T_\Gamma:=\prod_{v\in V(\Gamma)}\mathbb G_m, \ \ \ \widehat {A(\Gamma)} := \frac{k[[X_e,Y_e:e\in E(\Gamma)]]}{(X_eY_e: e\in E(\Gamma))} \ \ \text { and } \ \ \widehat
{B(\Gamma)} := \frac{k[[X_e,Y_e,T_e:e\in E(\Gamma)]]}{(X_eY_e-T_e: e\in E(\Gamma))}.
$$
Define an action of the  torus $T_\Gamma$ on $\widehat{A(\Gamma)}$ and $\widehat{B(\Gamma)}$  by the rule that $\lambda=(\lambda_v)_{v\in V(\Gamma)}\in T_\Gamma$ acts as
$$
\lambda \cdot X_e=\lambda_{s(e)}X_e\lambda_{t(e)}^{-1}, \ \ \  \lambda \cdot Y_e=\lambda_{t(e)}Y_e\lambda_{s(e)}^{-1}  \ \ \text{ and } \ \ \lambda \cdot T_e=T_e.
$$
Define complete local rings
$$
R_I:=\widehat{A(\Gamma)}[[W_1,\ldots,W_{g(X)-b_1(\Gamma)}]] \ \ \text{ and }  \ \ R_{(X,I)}=\widehat{B(\Gamma)}[[W_1,\ldots,W_{4g-3-b_1(\Gamma)-\# E(\Gamma)}]],
$$
with actions of $T_\Gamma$ induced by the actions on $\widehat{A(\Gamma)}$ and $\widehat{B(\Gamma)}$, and the trivial action on the remaining generators.

\begin{enumerate}[(i)]

	\item \label{Thm: MainThmA1} Suppose $\bar{J}(X)$ is  a compactified Jacobian of $X$.  If $[I]\in \bar{J}(X)$ with $I$ polystable, then there is an isomorphism
		\begin{displaymath}
			\widehat{\calO}_{\bar{J}(X), [I]} \cong R_I^{T_\Gamma}
		\end{displaymath}
		between the completed local ring of $\bar{J}(X)$ at $[I]$ and the $T_{\Gamma}$-invariant subring of $R_{I}$.

	\item \label{Thm: MainThmA2} If $X$ is a stable curve with trivial automorphism group and $[(X,I)]\in \pdbst$ with $I$ polystable,  
	 then there is an isomorphism
		\begin{displaymath}
			\widehat{\calO}_{\pdbst, [(X,I)]} \cong R_{(X,I)}^{T_{\Gamma}}
		\end{displaymath}
		between the completed  local ring of $\pdbst $ at $[(X,I)]$ and the $T_{\Gamma}$-invariant  subring of $R_{(X,I)}$.
	\end{enumerate}
\end{theoremalpha}

Theorem~\ref{Thm: MainThmA} is a consequence of Theorems~\ref{Thm: MainActionThm} and \ref{Thm: MainLocStr} (see also Remarks~\ref{Rem: TTorus}, \ref{remmtpf}).  We discuss the proof in more detail below.
The rings $\widehat {A(\Gamma)}$ appearing above are further studied in \cite{local2}.  In the notation of that paper, $\widehat {A(\Gamma)}$ is the completion of the ring
$A(\Gamma)$   defined in \cite[Theorem A]{local2}, and the action of $T_\Gamma$ in both papers is the same.  It is shown in \cite[Theorem A]{local2} that the invariant subring
$A(\Gamma)^{T_\Gamma}$ is isomorphic to the cographic ring $R(\Gamma)$ defined in \cite[Definition 1.4]{local2}.
In particular, the completed local ring of the compactified Jacobian is
 isomorphic to a power series ring over a completion of the cographic toric face ring $R(\Gamma)$.
A number of geometric properties of cographic rings are established in \cite{local2}, and some  consequences for compactified Jacobians are discussed in Theorem~\ref{Thm:
MainThmB} below.
The geometric and combinatorial properties of the rings $\widehat {B(\Gamma)}^{T_\Gamma}$ will be described in more detail by the authors  in \cite{CMKV3}.

\begin{theoremalpha} \label{Thm: MainThmB}
Let $\bar{J}(X)$ be a compactified Jacobian of a nodal curve $X$.
\begin{enumerate}[(i)]
\item \label{Thm: MainThmB1} $\bar{J}(X)$ has Gorenstein, semi log-canonical (slc) singularities. In particular,  $\bar{J}(X)$ is seminormal.
\item \label{Thm: MainThmB2} Let $[I]\in \bar{J}(X)$ with $I$ polystable. Then $[I]$ lies in the smooth locus of
$\bar{J}(X)$ if and only if $I$ fails to be locally free only at  separating nodes of the dual graph of $X$.
\end{enumerate}
\end{theoremalpha}

The proof is given at the end of \S\ref{Sec: Luna}.
In \cite{local2} it is shown that a number of further properties of cographic rings can be determined from elementary combinatorics of the graph $\Gamma=\Gamma_X(\Sigma)$
introduced in Theorem~\ref{Thm:  MainThmA}.   For instance, that paper provides combinatorial formulas giving the embedding
dimension and the multiplicity of  $\widehat{\calO}_{\bar{J}(X), x}$, as well as a description of the irreducible components
and the normalization of this ring.  The reader is directed to  \S\ref{secexam} and \cite{local2} for more details.
We also point out that it is well-known that the completed local ring of $\bar{J}(X)$ at a \emph{stable} sheaf is isomorphic to a completed product of nodes and smooth factors.  Using Theorem~\ref{Thm:
MainThmA} and the results of \cite{local2} one can construct examples of compactified Jacobians whose structure at a strictly semi-stable point is more complicated (see esp.~\S~\ref{sec2ir}).

We prove the theorems using  deformation theory.
The basic strategy is to show that the local structure of a compactified Jacobian is given by the subring of the mini-versal deformation ring that consists of elements invariant under an action of the automorphism group.  
Let us sketch the proof for a compactified  Jacobian $\bar{J}(X)$  of a nodal curve $X$ (the case of the  universal compactified Jacobian $\pdbst$ is similar).
To begin with, there is a well-known explicit description of the miniversal deformation ring $R_I$  of a rank $1$, torsion-free sheaf $I$ on a nodal curve $X$ (see Corollary~\ref{Cor: RingDesc}), and we use that description to define an  explicit action of  $\operatorname{Aut}(I)$ on $R_I$  (see Theorems~\ref{Thm: MainActionThm}).  
We prove the main result by showing that, when $[I]\in \bar{J}(X)$ with $I$ polystable, the ring of invariants  $R_I^{\operatorname{Aut}(I)}$ is isomorphic to the completed local ring of $\bar{J}(X)$ at $[I]$.

In order to establish this last point, we use the GIT construction of $\bar{J}(X)$ together with the Luna Slice Theorem and a theorem of Rim.
Recall that the compactified Jacobian is constructed as a GIT quotient of a suitable Quot scheme $\operatorname{Quot}(\calO_{X}^{\oplus r})$ by the action of
 $\SL_r$ (see Corollary \ref{C:GITpres}).  We check that the complete local ring $R_x$  of a slice (which exists by Luna Slice Theorem) 
at a polystable point $x=[\calO_X^{\oplus r}\twoheadrightarrow I]\in \operatorname{Quot}(\calO_{X}^{\oplus r})$ is a miniversal deformation ring for $I$ (Lemma~\ref{Lemma: SliceToRing}).  Thus $R_x$  is
(non-canonically) isomorphic to $R_I$. By definition, the stabilizer $G_x$ of $x$ (which is described in Lemma \ref{Lemma: StabToAut}) acts on the ring   $R_x$, and it follows from the definition of a slice that the invariant ring  $R_x^{G_x}$ is isomorphic to the complete local ring of the GIT quotient at the image of the point $x$.  We complete the proof by using a  theorem of Rim (Fact~\ref{Fact: Rim}) to identify the action of  $G_x$ on $R_x$ to our explicit  action of  $\operatorname{Aut}(I)$ on $R_I$, completing the proof (see Theorems.~\ref{Thm: MainActionThm} and \ref{Thm: MainLocStr}).

Theorem~\ref{Thm: MainThmB} is one consequence of Theorem~\ref{Thm: MainThmA}.  Other consequences will be found in the upcoming article \cite{CMKV3}.  There the authors will use Theorem~\ref{Thm: MainThmA} to describe the singularities of  $\pdbst$. More precisely, they will prove that $\pdbst$ has canonical singularities provided ${\rm char}(k)=0$.  When $X$ does not admit a non-trivial automorphism, the authors will prove this result by using the explicit description of the completed local ring in Theorem~\ref{Thm: MainThmA}, and in general, they will reduce the proof to a similar argument using a generalization of the Reid--Tai--Shepherd-Barron criterion for toric singularities. 
The results in \cite{CMKV3} will extend the work of Bini, Fontanari and the third author \cite{BFV}, where it is shown  that $\pdbst$ has canonical singularities when  $\gcd(d+1-g, 2g-2)=1$,  a condition  equivalent to the condition that $\pdbst$ has finite quotient singularities.  Under the same assumption on $d$ and $g$, the same authors  computed the  Kodaira dimension and the Itaka fibration of $\pdbst$ (\cite[Thm.~1.2]{BFV}), and in  \cite{CMKV3}, the present authors will extend that computation to all $d,g$.

 The authors also hope to use the results of this paper to study the singularities of the theta divisor of a nodal curve.  The theta divisor is an ample effective Cartier divisor on the canonical compactified Jacobian of degree $g-1$, parametrizing sheaves with a non-trivial section (see \cite{alex}, \cite{ctheta}, \cite{CV2}). The case of integral nodal curves has been studied by the first two authors in \cite{CMK}, where an analogue of the Riemann singularity theorem is proved.  The authors are currently investigating how to extend the Riemann singularity theorem to non-integral nodal curves, based upon the explicit local description of the compactified Jacobian obtained in this paper.

This paper suggests two technical questions for future study.  In Theorem~\ref{Thm: MainLocStr}\eqref{Thm: MainLocStr2} the curve $X$ is assumed to be automorphism-free.  It is particularly difficult to described the local structure of $\pdbst$ when $X$ admits an automorphism of order equal to $p$, the characteristic of $k$.  When $X$ admits such an automorphism, $\Aut(X, I)$ is reductive but not linearly reductive.  Linear reductivity is crucial in two places: in the proof of Theorem~\ref{Thm: MainLocStr}, which uses the Luna Slice Theorem, and Theorem~\ref{Thm: MainActionThm}, which uses a result of Rim.  It would be interesting to know if suitable generalizations of Rim's Theorem and the Luna Slice Theorem hold for reductive groups such as $\Aut(X,I)$.  We discuss these issues after the proofs of the two theorems.

Positive characteristic issues also appear in Fact~\ref{Fact: Comparison}, which relates the fibers of $\pdbst \to \Mgbar$ to compactified Jacobians.  That result is only stated in characteristic $0$, and it would be interesting to know if the result remains valid in positive characteristic.  This is discussed in greater detail immediately after the proof of the fact.

There are approaches to describing the local structure of a compactified Jacobian different from the approach taken here.  Alexeev has proven in \cite[Thm.~5.1]{alex} that  compactified Jacobians are
stable  semi-abelic varieties in the sense of \cite{alex1}, and consequently can be described using Mumford's construction \cite{Mum} of degenerations of abelian varieties.  In Mumford's approach (that has been further developed in \cite{nakamura}, \cite{Nam1}, \cite{AN99}, \cite{alex1}), one compactifies a semi-abelian variety by
first forming the projectivization of a (non-finitely generated) graded algebra and then quotienting out by a lattice.  This procedure provides direct access to the local structure of the
compactification, and thus Alexeev's work provides another approach to studying the local structure of compactified Jacobians.  It would be  interesting to compare the descriptions
arising from this approach to the descriptions given in this paper, but we do not pursue this topic here.

The results of  Theorem~\ref{Thm: MainThmB} are related to some results in the literature.  Specifically, it was known that $\bar{J}(X)$ is seminormal \cite[Thm.~5.1]{alex} and Gorenstein
\cite[Lemma~4.1]{AN99}.   In personal correspondence, Alexeev has explained to the  authors that the
techniques of those papers can also be used establish the fact that $\bar{J}(X)$  is semi-log canonical.
The description of the smooth locus of $\bar{J}(X)$ is certainly well-known to the experts (see e.g.~\cite[Thm.~6.1(3)]{caporaso}, \cite[Thm.~7.9(iii)]{cner}, \cite[Fact~4.1.5(iv)]{ctheta},
\cite[Fact~1.19(ii)]{MV}); however, it seems that a proof has not appeared in print.

This paper is organized as follows.  We review the definition and the GIT construction of (universal) compactified Jacobians in \S\ref{Sec: Prelim}
with the goal of collecting the facts needed to prove Theorems~\ref{Thm: MainThmA} and \ref{Thm: MainThmB}.
  The proofs of the main theorems begin in \S\ref{Sec: DefThy}, where we develop the deformation theory needed to compute
deformation rings parameterizing deformations of a rank $1$, torsion-free sheaf.  These rings admit
natural actions of automorphism groups, which are described in the next two sections.
The structure of the automorphism groups is studied in \S\ref{Sec: AutGrps}, and then those results are
used in \S\ref{Sec: ActionsOnRings} to compute group actions.  Finally, in \S\ref{Sec: Luna} we prove the main results of this paper by using the Luna Slice Theorem to relate the local
structure of a compactified Jacobian to a deformation ring.  In \S\ref{secexam} we describe some examples using results of \cite{local2}.

\subsection*{Acknowledgements}
We would like to thank Robert Lazarfeld for helpful expository suggestions.  This work began when the authors were visiting the MSRI, in Berkeley, for the special semester in
Algebraic Geometry in the spring of 2009; we would like
to thank the organizers of the program as well as the institute for the excellent working conditions and the stimulating atmosphere.

\subsection*{Conventions}
\begin{convention}
 	$k$ will denote an algebraically closed field (of arbitrary characteristic). All schemes are $k$-schemes, and all morphisms are implicitly assumed to respect the $k$-structure.
\end{convention}

\begin{convention}
	A \textbf{curve} is a connected, complete, reduced  scheme (over $k$) of pure dimension $1$. We denote by $\omega_X$ the \emph{dualizing sheaf} of $X$. The \emph{genus} $g(X)$ of a curve $X$ is $g(X) := h^1(X,\mathcal O_X)$.
\end{convention}

\begin{convention}
	A \textbf{subcurve} $Y$ of a curve $X$ is a  closed $k$-scheme $Y \hookrightarrow X$ that is reduced  and of pure dimension $1$ (but possibly disconnected). A subcurve $Y\subseteq X$ is said to be proper if it non empty and different from $X$.
Given a subcurve $Y$, the \emph{complementary subcurve} $Y^c$ is defined to be $\ov{X\setminus Y}$, or, in other words, $Y^c$ is the subcurve which is the union of all the irreducible components of $X$ that are not contained in $Y$.
\end{convention}

\begin{convention}
	A \textbf{family of curves} is a proper, flat morphism $X \to T$ whose geometric fibers are curves.
\end{convention}

\begin{convention}
	A \textbf{family of coherent sheaves} on a family of curves $X \to T$ is a $\calO_{T}$-flat, finitely presented $\calO_{X}$-module $I$.
\end{convention}

\begin{convention}
	A coherent sheaf $I$ on a curve $X$ is said to be:
 \begin{enumerate}[(i)]
\item  of \textbf{rank $1$} if $I$ has rank $1$ at every generic point;
\item \textbf{pure}  if for every non-zero
subsheaf $J\subseteq I$ the dimension of the support of $J$ is equal to the dimension of the support of $I$;
\item \textbf{torsion-free} if it is pure and the support of $I$ is $X$.
\end{enumerate}
The \textbf{degree} of a torsion-free, rank $1$ sheaf $I$ on a curve $X$ is defined to be
$\deg I:=\chi(I)-\chi(\calO_X)$, where $\chi$ denote the Euler characteristic.
\end{convention}

\section{Preliminaries on GIT and compactified Jacobians}  \label{Sec: Prelim}

Here we review the definition and the construction of compactified Jacobians of a fixed nodal curve as well as of the universal compactified Jacobian,  with the goal of collecting the results needed to prove Theorem~\ref{Thm: MainLocStr}.

\subsection{Geometric Invariant Theory} \label{Subsec: GIT}
 
 The (universal) compactified Jacobians are coarse moduli spaces of sheaves constructed  using Geometric Invariant Theory (GIT) and, in the proof of our results,  we will need to make use of their construction, and not just the  fact of their existence. Therefore, we will quickly review some background from GIT.

Recall that GIT is a tool for constructing a quotient of a quasi-projective variety $Q$ by the action of a reductive group $G$.  Given an auxiliary ample line bundle $\calO(1)$
together with a lift of the action of $G$ on $Q$ to an action on $\calO(1)$ (i.e.~a linearization), there is distinguished open subscheme $Q^{\text{ss}}$ of $Q$ that
consists of points that are semi-stable with respect to the
linearized action. The significance of $Q^{\text{ss}}$ is that it admits a \textbf{categorical quotient} that we define to be the GIT quotient of $Q$,
written $Q/\!\!/G$.   That is, there exists a pair $(Q^{\text{ss}}/G, \pi)$ consisting of a quasi-projective variety  $Q^{\text{ss}}/G$  and a $G$-invariant map
$$\pi \colon Q^{\text{ss}} \to Q^{\text{ss}}/G$$ 
with the property that	 $\pi$ is universal among all $G$-invariant maps out of $Q^{\text{ss}}$.
	When the characteristic of $k$ is $0$, the pair $(Q^{\text{ss}}, \pi)$ is actually a \textbf{universal categorical quotient}, i.e. for any morphism $T\to Q^{\text{ss}}/G$ the base change morphism
$\pi_T \colon Q^{\text{ss}}\times_{Q^{\text{ss}}/G} T  \to T$ is again a categorical quotient.

The local structure of $Q/\!\!/G$ is described by the Luna Slice Theorem, which compares $Q/\!\!/G$ to the quotient of a certain model $G$-space.
For the remainder of \S\ref{Subsec: GIT}, we assume that $Q$ is affine, so $Q = Q^{\text{ss}}$ and $Q /\!\!/ G$ is the categorical quotient   (\cite[Thm.  1, p.~27]{GIT}).
The model scheme is $G \times_{H} V$, whose definition we now review.  Suppose $H \subset G$ is
a reductive subgroup and $V$ a scheme with a left $H$-action.  The product $G \times V$ carries an
$H$-action defined by
\begin{displaymath}
	h \cdot (g,x) := (g h^{-1}, h \cdot x),
\end{displaymath}
and we write $G \times_{H} V$ for the categorical quotient.  This quotient admits a left
action of $G$ defined by the translation action on the first factor.
The two projections out of $G \times V$ induce morphisms
\begin{equation*}
	p \colon G \times_{H} V \to V/H  \hspace{0.5cm}  \text{ and }   \hspace{0.5cm} q \colon G \times_{H} V \to G/H.
\end{equation*}
The first map is $G$-invariant and realizes $V/H$ as the quotient of $G \times_{H} V$ by $G$.
The map $q$ is equivariant and can often be described as a contraction onto an orbit.  To be precise,
suppose we are given an element $v_0 \in V$ fixed by $H$.  One may verify that the image of
$(e, v_0) \in G \times V$ in $G \times_{H} V$ has stabilizer $H$, and the associated orbit map defines
a section of $q$.

The Luna Slice Theorem provides sufficient conditions for $Q/\!\!/G$ to be \'{e}tale locally isomorphic to  $V/H$ for a suitable $H$ and $V$.
More precisely,   let $x$ be a point of $Q$ with stabilizer $H$.  Given any affine, locally closed
subscheme $V \subset Q$ that contains $x$ and is stabilized by $H$ (i.e.~$H \cdot V \subset V$), the action map induces a $G$-equivariant morphism
$G \times_{H} V \to Q$.
We say that $V$ is a \textbf{slice} at $x$ if the following conditions are satisfied:
\begin{enumerate}
	\item the morphism $G \times_{H} V \to Q$ is \'{e}tale;
	\item the image of $G \times_{H} V \to Q$ is an open affine $U \subset Q$ that is $\pi$-saturated (i.e.~for each $u\in U$, $\pi^{-1}(\pi(u))\subseteq U$);

	\item the induced morphism $(G \times_{H} V) / G \to U / G$ is \'{e}tale;
	\item the induced morphism $G \times_{H} V \to U \times_{U/G} V/H$ is an isomorphism.
\end{enumerate}
Note in particular that condition (3) together with the observation above on the map $p$ implies that there is an \'etale morphism
\begin{equation}\label{E:sliceloc}
V/H\stackrel{\text{\'{e}t}}{\longrightarrow} Q/\!\!/G.
\end{equation}

The original Luna Slice Theorem \cite[p.~97]{lun} states that in characteristic zero a slice exists provided that $x$ is (GIT-)\textbf{polystable}, i.e.~the orbit of $x$ is closed.
When $x$ has a closed orbit, Matsushima's criterion implies that the stabilizer $H$ is reductive (\cite{Mat} for $k=\bbC$; \cite{Ric} for $k$ arbitrary).
Bardsley and Richardson have extended the Luna Slice Theorem to arbitrary characteristic.  With no assumptions on $\operatorname{char}(k)$, they prove
that a slice exists provided the orbit of $x$ is closed and the stabilizer $H$ is  reduced and linearly reductive (\cite[Prop.~7.6]{br}; the condition in \emph{loc.~cit.}~that the orbit is separable is
equivalent to our condition that $H$ is reduced).

\subsection{Compactified Jacobians of nodal curves}\label{SS:compJac}

In this subsection, we review the definition of compactified Jacobians of a nodal curve $X$.

Any (known) compactified Jacobian of $X$ parametrizes torsion-free, rank-$1$ sheaves on $X$ that are semistable with respect to some polarization. There are several ways to define a polarization and the associated semistability condition on $X$. The most general definition is stated in terms of numerical polarizations (following \cite[Sec. 2.4]{MV}): all the other know semistability conditions are special case of 
this one, see \cite[Sec. 2]{MV}.

\begin{defi}\label{D:numpol}
Let $X$ be a nodal curve with irreducible components $\{X_1,\ldots,X_{\gamma}\}$. A \textbf{numerical polarization} on $X$ is a $\gamma$-tuple of rational numbers
 $\phi=\{\phi_i=\phi_{X_i}\}_{i=1}^{\gamma}\in \bbQ^{\gamma}$, one for each irreducible component of $X$, such that $|\phi|:=\sum_i \phi_i\in \bbZ$.
For any subcurve $Y$ of $X$, we set $\phi_Y=\sum_{X_i\subseteq Y} \phi_{X_i}\in \bbQ$. For any subcurve  $Y\subset X$ such that $\phi_Y-\frac{\#(Y \cap Y^{c})}{2}\in \bbZ$, we define a numerical polarization $\phi^Y$ on $Y$ by setting
$$(\phi^{Y})_{Y_i}:=\phi_{Y_i} - \frac{\#(Y_i \cap Y^{c})}{2} \hspace{0.5cm} \text{ for any irreducible component } Y_i \text{ of } Y. $$
\end{defi}

The semistability of a torsion-free, rank $1$ sheaf on $X$ with respect to a numerical polarization $\phi$ is defined as it follows.

\begin{defi}\label{D:phi-ss}
Let $X$ be a nodal curve and let $\phi=(\phi_i)$ be a numerical polarization on $X$.
\begin{enumerate}[(i)]
\item \label{D:phi-ss1} A torsion-free, rank $1$ sheaf $I$ on $X$ is said to be \textbf{$\phi$-semistable} if $\deg I=|\phi|$ and
\begin{equation} \label{Eqn: PhiStable}
\deg(I_{Y}) \geq \phi_{Y} - \frac{\#(Y \cap Y^{c})}{2},
\end{equation}
for any subcurve $Y\subseteq X$, where $I_{Y}$ denotes the biggest torsion-free quotient of the restriction $I_{|Y}$ of $I$ to $Y$.
\item \label{D:phi-ss2} A torsion-free, rank $1$ sheaf $I$ on $X$ is said to be \textbf{$\phi$-stable} if it is $\phi$-semistable and the inequality \eqref{Eqn: PhiStable} is strict for every proper subcurve $\emptyset\neq Y\subsetneq X$.
\item \label{D:phi-ss3} A torsion-free, rank $1$ sheaf $I$ on $X$ is said to be \textbf{$\phi$-polystable} if it is $\phi$-semistable and for all subcurves $Y$ for which  inequality \eqref{Eqn: PhiStable} is an equality then it holds that $I=I_{Y}\oplus I_{Y^c}$.
\end{enumerate}
\end{defi}

The $\phi$-semistability condition is stated as a lower bound on the multidegree of $I$. However, this implies also an upper bound on the multidegree of $I$, as we observe in the following Remark.

\begin{remark}\label{R:phi-ss}
Let $\phi$ be a numerical polarization on a nodal curve $X$ and let $I$ be a torsion-free rank $1$ sheaf on $X$ of degree $\deg I=|\phi|$. Then
$I$ is $\phi$-semistable if and only if
\begin{equation} \label{Eqn: PrePhiStable}
	\deg(I_{Y}) \leq \phi_{Y} +\frac{\#(Y \cap Y^{c})}{2} - \#\{p \in Y \cap Y^{c} : \text{ $I$ fails to be locally free at $p$} \}.
\end{equation}
holds for every subcurve $Y$ of $X$. Indeed, this follows from the two easily checked formulas
$$ \begin{sis}
&\deg(I_{Y}) + \deg(I_{Y^{c}})= \deg I -  \#\{p \in Y \cap Y^{c} : \text{ $I$ fails to be locally free at $p$} \}, \\
& \phi_Y+\phi_{Y^c}=|\phi|.
\end{sis} $$
\end{remark}

The condition of being polystable is better understood in terms of the \emph{Jordan-H\"older}  filtration. Recall that, given a $\phi$-semistable sheaf $I$, a Jordan-H\"older filtration of $I$ is a filtration
$$0=I_{q+1}\subsetneq I_q\subsetneq \ldots \subsetneq I_1\subsetneq I_0=I,$$
with the following properties:
\begin{enumerate}\label{E:JH}
\item \label{E:JH1}  The sheaf $I_k$ is a rank $1$, torsion-free sheaf supported on a subcurve $Z_k\subset X$ and $\phi^{Z_k}$-semistable, for every $0\leq k\leq q$.
\item \label{E:JH2} The quotient sheaf $I_k/I_{k+1}$  is a rank $1$, torsion-free sheaf supported on the subcurve $Y_k=Z_k\setminus Z_{k+1}$ and $\phi^{Y_k}$-stable, for every $0\leq k\leq q$.
\end{enumerate}
Jordan-H\"older filtrations exist for every $\phi$-semistable sheaf $I$ but they are not unique; however, the graded sheaf 
$$\Gr(I):=I_0/I_1\oplus \ldots \oplus I_q/I_{q+1}$$
depends only on $I$ (see e.g. \cite[\S 1.3]{esteves01}, \cite[\S 2.5]{MV}). 
Then it is easy to check that $I$ is polystable if and only if $I\cong \Gr(I)$. Moreover, we say that two $\phi$-semistable sheaves $I$ and $I'$ are \emph{S-equivalent} (or Jordan-H\"older equivalent)
if $\Gr(I)\cong \Gr(I')$. Therefore, every $\phi$-semistable sheaf is S-equivalent to a unique $\phi$-polystable sheaf, namely $\Gr(I)$.

With the above definitions, we can now introduce the \textbf{$\phi$-compactified Jacobian functor}
$$\bar{J}^{\sharp}_{\phi}(X) \colon \text{$k$-Sch} \to \Sets$$
which associates to a $k$-scheme $T$ the set of families of coherent sheaves on $X\times_k T\to T$ that are fiberwise rank $1$, torsion-free and $\phi$-semi-stable.

\begin{fact}[Oda-Seshadri, Seshadri]\label{F:OSJac}
There exists a projective variety $\bar{J}_{\phi}(X)$, called the \textbf{$\phi$-compa\-ctified Jacobian} or simply \textbf{compactified Jacobian},
that co-represents the functor $\bar{J}_{\phi}^{\sharp}(X)$. Moreover, two sheaves $I, I'\in \bar{J}_{\phi}^{\sharp}(X)(k)$ define the same $k$-point $[I]=[I']\in \bar{J}_{\phi}(X)$ if and only if $I$ and $I'$ are S-equivalent. In particular, every $k$-point of $\bar{J}_{\phi}(X)$ is equal to $[I]$, for a unique $\phi$-polystable sheaf $I$.  
\end{fact}
Recall that the fact that  $\bar{J}_{\phi}(X)$ co-represents $\bar{J}_{\phi}^{\sharp}(X)$ means that there exist a natural transformation of functors $\pi:\bar{J}_{\phi}^{\sharp}(X)\to \Hom(-, \bar{J}_{\phi}(X))$ which is universal with respect to natural transformations from $\bar{J}_{\phi}^{\sharp}(X)$ to the functor of points of $k$-schemes. Given a point $I \in \bar{J}_{\phi}^{\sharp}(X)(k)$, we set 
$[I]:=\pi(I)\in \Hom(\Spec k,  \bar{J}_{\phi}(X))= \bar{J}_{\phi}(X)(k)$.

\begin{proof}
This is proved by Oda-Seshadri  in \cite[Thms~11.4 and 12.14]{OS} and Seshadri in \cite[Thm. 15]{Ses}.
Note that in \emph{loc.~cit.}~the authors use two different definitions of $\phi$-(semi)stability, which are however equivalent to our definition as discussed in  \cite[\S 2.1]{alex} and \cite[Sec. 2]{MV}.
\end{proof}

\begin{remark}\label{R:gen-pola}
If the numerical polarization $\phi$ is such that $\phi_{Y} - \frac{\#(Y \cap Y^{c})}{2}\not\in \bbZ$ for every proper subcurve $\emptyset\neq Y\subsetneq X$ (in which case we say that $\phi$ is \emph{general}), 
then it follows from Definition \ref{D:phi-ss} that every $\phi$-semistable sheaf is also $\phi$-stable. Hence, $\bar{J}_{\phi}(X)$ is a fine moduli space  parametrizing $\phi$-stable sheaves and we say that 
$\bar{J}_{\phi}(X)$ is a \emph{fine compactified Jacobian}.  Such compactified Jacobians are studied in \cite{MV} and in \cite{MRV}.
\end{remark}

We now compare the $\phi$-semistability condition introduced above with the (more familiar) notion of slope semistability. Recall that, given a nodal curve $X$ and a polarization $L$ on $X$, i.e. an ample line bundle on $X$, the \textbf{slope} $\mu_{L}(I)$ of a coherent sheaf $I$ with respect to $L$ is defined to be $a/r$, where $a$ and $r$ are coefficients of the Hilbert polynomial $P_L(I,t) := r \cdot t + a$ of $I$ with respect to $L$.

\begin{defi}\label{D:slopess}
Let $X$ be a nodal curve and $L$ be a polarization on $X$.
\begin{enumerate}[(i)]
\item The sheaf $I$ is said to be \textbf{slope semistable} (resp. \textbf{slope stable}) with respect to the polarization $L$ if it is pure and satisfies $\mu_L(I) \le \mu_L(J)$ (resp. $<$) for all pure non-trivial quotients $I \twoheadrightarrow J$ with $1$-dimensional support $\operatorname{Supp}(J)$.
\item  The sheaf $I$ is said to be \textbf{slope polystable} if it is slope semistable and  isomorphic to a direct sum of slope stable sheaves.
\end{enumerate}
\end{defi}

With the above definitions, we can now introduce the \textbf{Simpson Jacobian functor}
of degree $d$ to be the functor
\begin{displaymath}
	\bar{J}_{L,d}^{\sharp}(X) \colon \text{$S$-Sch.} \to \Sets
\end{displaymath}
which sends a $k$-scheme $T$ into the set of families of coherent sheaves on $X\times_k T\to T$ that are fiberwise rank $1$, torsion-free of degree $d$ and slope semistable with respect to the polarization $L$.

\begin{fact}[Simpson]\label{F:SimpJac}
There exists a projective scheme $\bar{J}_{L,d}(X)$, called the \textbf{Simpson compactified Jacobian},
that co-represents the functor $\bar{J}_{L,d}^{\sharp}(X)$.
\end{fact}
\begin{proof}
This follows easily from the work of Simpson \cite{simpson}. However, for later use, we need to review the explicit GIT construction.
Consider the polynomial
\begin{equation*}
	P_{d}(t) := \deg(L)\cdot  t + d + 1 - g(X),
\end{equation*}
which is the Hilbert polynomial, with respect to the polarization $\calO_X(1):=L$, of any rank $1$, torsion-free sheaf of degree $d$. Let $\operatorname{M}^{\sharp}(\calO_{X}, P_d) \colon \text{$k$-Sch.} \to \Sets $ to be the functor which associates to a $k$-scheme $T$ the set of isomorphism classes of coherent sheaves on $X\times_k T$, flat over $T$, that are fiberwise slope semistable with Hilbert polynomial $P_d(t)$ with respect to $L$. Note that $\bar{J}_{d,L}^{\sharp}(X)$ is the subfunctor of $\operatorname{M}^{\sharp}(\calO_{X}, P_d)$ parametrizing families of torsion-free, rank $1$ sheaves.

Following Simpson's construction \cite{simpson} (note that Simpson \cite{simpson} works over $k=\bbC$ but his construction has been extended over an arbitrary base field $k$ by Maruyama \cite{maruyama} and Langer \cite{langer}), choose $b$ sufficiently large and set $r := P_d(b)=d\cdot b+1-g(X)$. Consider the Quot scheme $\Quot(\calO_{X}(-b)^{\oplus r}, P_d(t))$ parametrizing quotients $\calO_{X}(-b)^{\oplus r}\twoheadrightarrow I$, where $I$ is a coherent sheaf on $X$ of Hilbert polynomial $P_d(t)$ with respect to $\calO(1)$ (see \cite{gro} and \cite{nitsure} for details on Quot schemes).

There is  a closed and open subscheme (\cite[p.~66]{simpson})
$Q^\circ\subseteq \Quot(\calO_{X}(-b)^{\oplus r}, P_d)$ that parameterizes quotient maps
\begin{displaymath}
	q \colon \calO_{X}(-b)^{\oplus r} \twoheadrightarrow I
\end{displaymath}
satisfying the following additional conditions:
\begin{itemize}
	\item $H^{1}(X, I(b))=0$;
	\item $q \otimes 1 \colon H^{0}(X, \calO_{X}^{\oplus r}) \to H^{0}(X,I(b))$ is an isomorphism;
	\item $I(b)$ is generated by its global section.
\end{itemize}
The natural linearized action of $\operatorname{SL}_{r}$ on the Quot scheme $\Quot(\calO_{X}(-b)^{\oplus r}, P_d(t))$  restricts to a linearized action on $Q^\circ$ and the GIT stability for this action is naturally related to slope stability.
Specifically,  a point of the Quot scheme corresponding to  $q \colon \calO(-b)^{\oplus r}
\twoheadrightarrow I$ is GIT (resp. semi, poly)stable if and only if $I$ is (resp.
semi, poly)stable with respect to the polarization $L$ (see \cite[Cor. 1.20, Thm. 1.19, Pf. of Thm. 1.21]{simpson}).
Therefore, the projective GIT quotient $\operatorname{M}(\calO_{X}, P_d):=Q^\circ/\!\!/\operatorname{SL}_r=(Q^\circ)^{\rm ss}/SL_r$ naturally co-represents the functor $\operatorname{M}^{\sharp}(\calO_{X}, P_d)$.

Consider now the locus $Q^{\circ\circ}\subseteq Q^\circ$
parametrizing quotients $q \colon \calO_{X}(-b)^{\oplus r} \twoheadrightarrow I$ such that
\begin{itemize}
\item $I$ is a rank $1$, torsion-free sheaf on $X$.
\end{itemize}
This is a $\operatorname{SL}_{r}$-invariant subset that is closed and open in $Q^{\circ}$ by \cite[Lemma~8.1.1]{Pan}. Therefore, the image of $Q^{\circ\circ}$ in the GIT quotient $Q^\circ/\!\!/\operatorname{SL}_r$, which we set to be equal to $\bar{J}_{L, d}(X/S)$, must be closed-and-open in $\operatorname{M}(\calO_{X}, P_d)$ by \cite[p.~8, Remark~6]{GIT}.  By construction,
the projective scheme $\bar{J}_{L, d}(X/S)$ co-represents the functor  $\bar{J}_{L, d}^{\sharp}(X/S)$.
\end{proof}

Simpson compactified Jacobians are a special case of Oda-Seshadri $\phi$-compactified Jacobians.

\begin{fact}[Alexeev]\label{F:compJac}
Let $X$ be a nodal curve endowed with a polarization $L$ and fix $d\in \bbZ$.
Consider the numerical polarization $\phi$ such that
 \begin{equation}\label{E:phi-L}
 \phi_{X_i}:=\frac{\deg(L|_{X_i})}{\deg(L)}\left( d - \frac{\deg(\omega_{X})}{2}\right)
		+ \frac{\deg(\omega_{X}|_{X_i})}{2},
\end{equation}
for each irreducible component $X_i$ of $X$. Then a rank $1$, torsion-free sheaf $I$ of degree $d$ on $X$ is slope semistable (resp. stable, resp. polystable) with respect to $L$ if and only if it is $\phi$-semistable
(resp. $\phi$-stable, $\phi$-polystable).

In particular, we have that $\bar{J}_{L,d}^{\sharp}(X)=\bar{J}_{\phi}^{\sharp}(X)$, which implies
$\bar{J}_{L,d}(X)\cong \bar{J}_{\phi}(X)$.
\end{fact}
\begin{proof}
This is proved by Alexeev in \cite{alex}, where it shown that a torsion-free, rank $1$ sheaf $I$ of degree $d$ is slope (semi)stable (with respect to the polarization $L$) if and only if, for any subcurve $i \colon Y \hookrightarrow X$, we have that
$\mu_L(I) \le \mu_L(i_{*}(I_{Y}))$, where $I_Y$ is the biggest torsion-free quotient the restriction $i^*(I)=I_{|Y}$. 
By the definition of the slope $\mu_L$, we get 
\begin{equation} \label{E:slope1}
	\frac{\deg(I) - 1/2 \deg(\omega_{X})}{\deg(L)}=\mu_L(I) \le \mu_L(i_{*}(I_{Y}))=  \frac{\deg(I_{Y}) -1/2 \deg(\omega_{X}|_{Y}) + 1/2 \, \#(Y \cap Y^{c})}{\deg(L|_{Y})},
\end{equation}
where we used the formula $\omega_{X}|_{Y} = \omega_{Y}(Y\cap Y^c)$. 
Equation \eqref{E:slope1} can be rewritten as
\begin{equation} \label{E:slope2}
	\deg(I_{Y}) \geq   \frac{\deg(L|_{Y})}{\deg(L)}\left( \deg(I) - \frac{\deg(\omega_{X})}{2}\right)
		+ \frac{\deg(\omega_{X}|_{Y})}{2} -  \frac{\#(Y \cap Y^{c})}{2}=\phi_Y-  \frac{\#(Y \cap Y^{c})}{2},
\end{equation}
which says that $I$ is $\phi$-(semi)stable with respect to the numerical polarization defined by \eqref{E:phi-L}. 
The fact that slope polystability correspond to $\phi$-polystability follows easily from the above. 
\end{proof}

\begin{remark}\label{R:OS-Simp}
Let $X$ be a nodal curve of genus $g$ and consider the compactified Jacobians of $X$.
\begin{enumerate}[(i)]
\item \label{R:OS-Simp1} \emph{There are  $\phi$-compactified Jacobians  of degree $d$ that are not Simpson Jacobians of degree $d$.}

The most extreme case is $d=g-1$.  As it follows from \eqref{E:phi-L}, every Simpson compactified Jacobian of degree $g-1$ is isomorphic to the $\phi$-compactified Jacobian $\bar{J}_{\phi}(X)$ such that 
$\displaystyle \phi_{X_i}=\frac{\deg({\omega_X}_{|X_i})}{2}$ for every irreducible component $X_i$ of $X$ (a very special compactified Jacobian, called the canonical compactified Jacobian of degree $g-1$,  that was studied in detail in \cite[Sec. 3]{alex}, \cite{ctheta}, \cite{CV2}).
 However, there are many $\phi$-compactified Jacobians of degree $d=g-1$ (indeed, as many as in the other degrees).

Also in degree $d\neq g-1$, there are, in general, $\phi$-compactified Jacobians that are not Simpson compactified Jacobians. For example, let $X$ be the genus $2$ nodal curve
that consists of two rational components meeting in three nodes.  For a numerical polarization $\phi = (\phi_1, \phi_2)$ such that $|\phi|=\phi_1+\phi_2=0$, then one can compute that
$\bar{J}_{\phi}(X)$ has two irreducible components if $\phi_1, \phi_2 \in 1/2 + \bbZ$ and three irreducible components otherwise. On the other hand,  given an
ample line bundle $L$ with bidegree $(a,b)$, the associated $\phi$-parameter (see \eqref{E:phi-L}) is
\begin{displaymath}
	\phi = ( 1/2 - b/(a+b), 1/2 - a/(a+b)),
\end{displaymath}
and $\phi_1,\phi_2$ cannot belong to $1/2+\bbZ$ because $a, b>0$. Therefore, every Simpson compactified Jacobian of degree $0$ has three irreducible components.

\item \label{R:OS-Simp2} \emph{Every $\phi$-compactified Jacobian  of degree $d$ is isomorphic to a  Simpson Jacobian of degree $d'$, for some $d'\gg d$.}

 Indeed given a numerical polarization $\phi$, pick a line bundle $M$ of sufficiently small degree on each irreducible component $X_i$ of $X$ in such a way that
$$
	a_i:=\phi_{X_i} -\deg(M|_{X_i}) - \frac{\deg(\omega_{X}|_{X_i})}{2} > 0.
$$
Moreover, pick a sufficiently divisible natural number $e\in \bbN$ such that
$$
	b_i:=e\frac{a_i}{d+g-1}\in \bbZ \: \text{ for every } i.
$$
Finally, choose a line bundle $L$ of total degree $e$ such that $\deg(L_{|X_i})=b_i$
and observe that $L$ is ample since $\deg(L_{|X_i})=b_i>0$.
With the above choices, we get that
\begin{equation} \label{Eqn: LMToPhi}
	\psi_{X_i}:=\phi_{X_i} -\deg(M|_{X_i}) = \left(d- \frac{\deg(\omega_{X})}{2} \right) \frac{\deg(L|_{X_i})}{\deg(L)}+\frac{\deg(\omega_{X}|_{X_i})}{2}.
\end{equation}
Therefore, using Fact \ref{F:compJac}, we get the isomorphism
\begin{equation*}
\begin{aligned}
\bar{J}_{\phi}(X) & \stackrel{\cong}{\longrightarrow} \bar{J}_{\psi}(X)\cong \bar{J}_{L,d'}(X)\\
I & \mapsto I\otimes M^{-1},
\end{aligned}
\end{equation*}
where $d'=|\psi|=|\phi|-\deg M$.
\end{enumerate}
\end{remark}

We record in the following corollary a presentation of any compactified Jacobian of a nodal curve as a GIT quotient of an open subset of a suitable Quot scheme. Such a GIT description will be crucial in proving Theorem \ref{Thm: MainThmA}\eqref{Thm: MainThmA1}.

\begin{cor}[GIT presentation of compactified Jacobians]\label{C:GITpres}
Let $X$ be a nodal curve of genus $g$ and let $\bar J(X)$ be any compactified Jacobian of $X$.
There exists a Quot scheme $\Quot(\calO_X^{\oplus r},P_d(t))$, parametrizing quotients
$q:\calO_X^{\oplus r} \twoheadrightarrow I$ with Hilbert polynomial $P_d(t)=d\cdot t+1-g$ with respect to some ample line bundle $\calO_X(1)$, with an open and closed $\SL_r$-invariant subscheme $U \subseteq \Quot(\calO_X^{\oplus r},P_d(t))$ parameterizing the quotients $q \colon \calO_{X}^{\oplus r} \twoheadrightarrow I$ with the property that
\begin{enumerate}
	\item \label{C:GITpres1} $H^{1}(X, I)=0$,
	\item \label{C:GITpres2} $q \colon H^{0}(X,\calO_{X}^{\oplus r}) \to H^{0}(X, I)$ is an isomorphism,
	\item \label{C:GITpres2b} $I$ is generated by the global sections,
	\item \label{C:GITpres3} $I$ is a torsion-free, rank $1$  sheaf,
\end{enumerate}
in such a way that
$$\bar{J}(X)\cong U/\!\!/\operatorname{SL}_r=U^{\rm ss}/\SL_r,$$
where the GIT quotient on the right hand side is taken with respect to the natural linearized action of $\SL_r$.
\end{cor}
\begin{proof}
By Remark \ref{R:OS-Simp}\eqref{R:OS-Simp2}, it is enough to prove the Corollary for a Simpson compactified Jacobian $\bar{J}_{L,e}(X)$ (with $e\gg0$). In the proof of Fact \ref{F:SimpJac}, we have seen that $\bar{J}_{L,e}(X)$ admits a GIT description as $Q^{\circ\circ}/\!\!/\operatorname{SL}_r$, where $Q^{\circ\circ}$ is the open and closed subscheme of a suitable Quot scheme $\Quot(\calO_X(-b)^{\oplus r}, P_e(t))$, parametrizing the quotients
$q:\calO_X(-b)^{\oplus r}\twoheadrightarrow I$ such that
\begin{itemize}
\item $H^{1}(X, I(b))=0$;
\item $q \otimes 1 \colon H^{0}(X, \calO_{X}^{\oplus r}) \to H^{0}(X,I(b))$ is an isomorphism;
\item $I$ is a rank $1$, torsion-free sheaf on $X$.
\end{itemize}
The isomorphism
\begin{equation*}
\begin{aligned}
\Phi: \Quot(\calO_X(-b)^{\oplus r}, P_e(t)) & \stackrel{\cong}{\longrightarrow} \Quot(\calO_X^{\oplus r},P_{e+b\deg \calO_X(1)}(t)), \\
[q:\calO_X(-b)^{\oplus r}\twoheadrightarrow I] & \mapsto [q:\calO_X^{\oplus r}\twoheadrightarrow I(b)],
\end{aligned}
\end{equation*}
sends $Q^{\circ\circ}$ isomorphically onto the open subset $U\subseteq \Quot(\calO_X^{\oplus r},P_{e+b\deg \calO_X(1)}(t))$ parametrizing quotients $q:\calO_X^{\oplus r}\twoheadrightarrow I$ satisfying the three conditions \eqref{C:GITpres1}, \eqref{C:GITpres2}, \eqref{C:GITpres3} and moreover
$$\bar{J}_{L,e}(X)\cong Q^{\circ\circ}/\!\!/\operatorname{SL}_r \stackrel{\cong}{\longrightarrow}
U /\!\!/\operatorname{SL}_r.$$
\end{proof}

\subsection{The universal compactified Jacobian}\label{SS:univJac}

In this subsection, we review the definition and the construction of the universal degree $d$ compactified Jacobian $\bar J_{d,g}\to \ov{M}_g$ over the moduli space of stable curves $\ov{M}_g$ of genus $g\geq 2$.

This construction of $\bar J_{d,g}$ is originally due to Caporaso \cite{caporaso} in terms of balanced line bundles on quasi-stable curves. Later, Pandharipande \cite{Pan} re-interpreted $\bar J_{d,g}$ in terms of rank $1$, torsion-free semistable sheaves on stable curves.  We will focus on Pandharipande's later construction because this description  most naturally relates to the other
compactified Jacobians we discuss here. For a description of Caporaso's approach, we direct the interested reader to \cite{caporaso} and \cite[\S10]{Pan}.

Given integers $d$ and  $g \ge 2$, the \textbf{universal compactified Jacobian functor}
$$\pdbstFunc \colon \text{$k$-Sch.} \to \Sets$$
is defined to be the  functor sending a $k$-scheme $T$ to the set of isomorphism classes of families $\calX\to T$ of stable curves of genus $g$ together with a family of coherent sheaves 
which is fiberwise torsion-free, rank $1$ of degree $d$ and slope semistable with respect to the relative dualizing line bundle.

\begin{fact}[Pandharipande \cite{Pan}]\label{F:univJac}
The functor  $\pdbstFunc$ is co-representable by a projective scheme $\pdbst$, called the \textbf{universal compactified Jacobian}, which is endowed with a forgetful projective morphism $\Phi:\bar J_{d,g}\to \ov{M}_g$.
\end{fact}

\begin{proof}
This follows from the work of Pandharipande \cite{Pan}, where the projective scheme $\bar J_{d,g}$
is constructed via GIT.
Since we will need this GIT description in the proof of Theorem \ref{Thm: MainThmA}\eqref{Thm: MainThmA2}, we will now review the relevant GIT set-up.

To begin, we may assume $d$ is sufficiently large because tensoring with
the dualizing sheaf defines a canonical isomorphism between $\pdbstFunc$ and
$\overline{J}^{\sharp}_{d+2g-2,g}$. Thus, let
$d$ be large and fixed.  Set $N := 10 (2 g -2) -g$ and $e := 10(2g-2)$. Consider the polynomial $P(t)  := e \cdot t + d + 1-g$ and set $r := P(0)$.

Inside of the Hilbert scheme of degree $e$ curves in $\bbP^N$, we can consider the locally closed subscheme $H_{g}$
parameterizing non-degenerate, $10$-canonically embedded stable curves.
The product $H_{g} \times \bbP^{N}$ contains the universal $10$-canonically embedded curve $X_{g}$,  and associated to this family is the relative Quot scheme $\operatorname{Quot}(\calO_{X_g}^{\oplus r}, P(t))$, parametrizing quotients $q:\calO_{X_g}^{\oplus r}\twoheadrightarrow E$ such that $E$ is a coherent sheaf on $X_g$, flat over
$H_g$, with the property that on each fiber of $X_g\to H_g$ the Hilbert polynomial (with respect to the polarization given the embedding $X_g\hookrightarrow H_g\times \bbP^N$) of $E$ is equal to $P(t)$.

The product group $\operatorname{SL}_{r} \times \operatorname{SL}_{N+1}$ acts on this Quot scheme
by making $\operatorname{SL}_{r}$ act on $\calO_{X_g}^{\oplus r}$ by changing bases,
$\operatorname{SL}_{N+1}$ act on $\bbP^{N}$ by changing projective coordinates, and
then making $\operatorname{SL}_{r} \times \operatorname{SL}_{N+1}$ act on the Quot scheme by the product action.  The action of $\operatorname{SL}_{r}
\times \operatorname{SL}_{N+1}$ admits a natural linearization coming from the construction of the relative Quot scheme (see  \cite{nitsure}).

Inside $\operatorname{Quot}(\calO_{X_g}^{\oplus r}, P(t))$, there  is an invariant closed-and-open
subset $\calQ^\circ$ parameterizing torsion-free, rank $1$ quotients (\cite[Lemma~8.1.1]{Pan}). It is shown in \cite[Thm. 8.2.1, Thm. 9.1.1]{Pan} that a point $[q:\calO_{X_g}^{\oplus r}\twoheadrightarrow E]\in \calQ^{\circ}$ is GIT semistable if and only if $E$ is relatively semistable with respect to the relative dualizing sheaf. Therefore, the GIT (projective) quotient 
\begin{equation}\label{E:GIT-unJac}
\pdbst:=\calQ^\circ /\!\!/ \operatorname{SL}_{r}\times \operatorname{SL}_{N+1}.
\end{equation}
co-represents $\pdbstFunc$ and, by construction, it is endowed with a forgetful projective morphism $\Phi:\pdbst \to \Mgbar$.
\end{proof}

The fibers of the forgetful morphism $\Phi:\pdbst\to \Mgbar$ are related to compactified Jacobians of stable curves with respect to their canonical polarization.

\begin{fact} \label{Fact: Comparison}
	Assume ${\rm char}(k) = 0$.  Then the fiber of $\Phi:\pdbst\to \Mgbar$ over a stable curve $X$ is 
	\begin{equation}
\Phi^{-1}(X)\cong \bar{J}_{\omega_X,d}(X)/\Aut(X).
\end{equation}
\end{fact}
\begin{proof}
This fact is surely well-known (see e.g. \cite[\S 1.8]{alex}); however, we sketch a proof for the lack of a suitable reference.
 
Since ${\rm char}(k)=0$, the GIT quotient \eqref{E:GIT-unJac} is a universal categorical quotient (see \S\ref{Subsec: GIT}), which implies that the  variety $\pdbst$ co-represents the functor $\pdbstFunc$ \emph{universally}, i.e. for any scheme $T\to \pdbst$ the base change functor $\pdbstFunc\times_{\Hom(-,\pdbst)}\Hom(-,T)$ is co-represented by $T$. Applying this property to the inclusion $\Phi^{-1}(X)\hookrightarrow \pdbst$, we deduce that
$\Phi^{-1}(X)$ co-represents the functor $$(\pdbstFunc)_{|X}:=\pdbstFunc\times_{\Hom(-,\pdbst)}\Hom(-,\Phi^{-1}(X))\colon \text{$k$-Sch.} \to \Sets$$
which associates to a $k$-scheme $S$ the set of 
isomorphism classes of iso-trivial families $p:\calX\to S$ 
with fiber $X$ together with a coherent sheaf $\calI$, flat over 
$S$, which is fiberwise
torsion-free, rank $1$ and $\omega_{\calX/S}$-semistable of degree $d$.
Therefore, there is a natural transformation of functors $\bar{J}^{\sharp}_{\omega_X,d} \to (\pdbstFunc)_{|X}$ which factors through the natural action of $\Aut(X)$ on $\bar{J}^{\sharp}_{\omega_X,d}$
\begin{equation}\label{E:nat-tra}
\eta: \bar{J}^{\sharp}_{\omega_X,d}/\Aut(X) \to (\pdbstFunc)_{|X}.
\end{equation}
It is easily since that $\eta$ is a local isomorphism in the \'etale topology (using that every iso-trivial family becomes trivial after an \'etale base change); therefore, passing to the varieties that co-represent the above functors, we get an isomorphism
\begin{equation}\label{E:iso-corep}
\ov{\eta}: \bar{J}_{\omega_X,d}/\Aut(X) \stackrel{\cong}{\longrightarrow} \Phi^{-1}(X).
\end{equation}
\end{proof}

\begin{remark}
If ${\rm char}(k)\gg g$, then the stabilizers of the GIT quotient \eqref{E:GIT-unJac} are linearly reductive (by Lemma \ref{Lemma: StabToAut}\eqref{Lemma: StabToAut2} and Corollary  \ref{Cor: GrpRLinRed}), which implies that the above GIT quotient is a universal categorical quotient, and so the proof of Fact \ref{Fact: Comparison} still goes through.
We ignore the question of whether, in small characteristic,   the GIT quotient \eqref{E:GIT-unJac} remains universal  and  Fact \ref{Fact: Comparison} still holds true.
\end{remark}

\section{Deformation Theory} \label{Sec: DefThy}
In the previous section, we studied the representability properties of global moduli functors parameterizing rank $1$, torsion-free sheaves on a (fixed or varying) nodal curve.  
This section focuses on the  analogous local topic: the pro-representability properties of deformation functors parameterizing infinitesimal deformations of a rank $1$, torsion-free sheaf on a (fixed or varying) nodal curve.  The main result is Corollary~\ref{Cor: RingDesc}, which
explicitly describes miniversal deformations rings parameterizing such deformations. The corollary is used in \S\ref{Sec: Luna} to prove  Theorem~\ref{Thm: MainThmA} by relating the deformation rings to the completed local rings of (universal) compactified Jacobians (Thm.~\ref{Thm: MainLocStr}).

\subsection{The deformation functors} \label{Subsec: DefFunc}
We begin by reviewing the deformation functors of interest. 
 
\begin{defi} \label{Def: DeformDef}
	Suppose we are given a $k$-scheme $S$, a finitely presented $\calO_{S}$-module $F$, and
	a local $k$-algebra $A$ with  residue field $k$.  A \textbf{deformation of the pair} $(S, F)$ over $A$ is  a quadruple $(S_A, F_A, i, j)$ that consists of
	\begin{enumerate}
	
		\item a flat $A$-scheme $S_A$;
		\item a $A$-flat, finitely presented $\calO_{S_{A}}$-module $F_A$;
		\item an isomorphism $i \colon S_A \otimes_{A} k \isoto S$;
		\item an isomorphism $j \colon i_{*}(F_A \otimes_{A} k) \isoto F$ of $\calO_{S}$-modules.
	\end{enumerate}
	The \textbf{trivial deformation} of a pair $(S, F)$ over $A$ is defined to be the quadruple $(S \otimes_{k} A, F \otimes_{k} A,  i_{\text{can}}, j_{\text{can}})$.  Here $i_{\text{can}}
$ and $j_{\text{can}}$ are defined to be the canonical maps.
	If $(S'_A, F'_A, i', j')$ is a second deformation of the pair $(S, F)$, then an \textbf{isomorphism} from $(S_A, F_A, i, j)$ to $(S'_A, F'_A, i', j')$ is defined to be a pair $(\phi, \psi)$
that consists of

	\begin{enumerate}
		\item an isomorphism $\phi \colon S_{A} \isoto S'_{A}$ over $A$ such that $i'\circ (\phi\otimes 1)=i$;
		\item an isomorphism $\psi \colon \phi_{*}(F_A) \isoto F'_{A}$ of $\calO_{S_{A'}}$-modules such that $j'\circ i'_*(\psi \otimes 1)=j$.
	\end{enumerate}
 
	A \textbf{deformation of the scheme} $S$ over $A$ is defined by omitting the data of $F_{A}$ and $j$ from the definition of a deformation of a pair.  Similarly, an
\textbf{isomorphism} from one deformation $(S_A, i)$ of $S$ to another $(S'_{A}, i')$ is defined by omitting $\psi$ from Definition~\ref{Def: DeformDef}.  The scheme $S$ always
admits the \textbf{trivial deformation} over $A$ given by the pair $(S \otimes_{k} A, i_{\text{can}})$.
	
	A \textbf{deformation of a sheaf $F$} over $A$ is defined to be  a pair $(F_A, j)$ such that the quadruple $(S \otimes_{k} A, F_A, i_{\text{can}}, j)$ is a deformation of the pair $(S,
F)$.  An \textbf{isomorphism} from one deformation of $F$ to another is defined to be a deformation of the associated deformations of the pair $(S, F)$.  The trivial deformation of the
pair $(S, F)$
	may be considered as a \textbf{trivial deformation} of $F$.
\end{defi}

Let  $\art$ be the category of artin local  $k$-algebras with residue field $k$.
Recall that a \textbf{deformation functor}
is a functor $F \colon \art \to \sets$ of artin rings with the property that $F(k)$ is a singleton set.
We study the following deformation functors.
\begin{defi} \label{Def: DeformFunc}
	Define functors $\Def_{S}, \Def_{F}, \Def_{(S, F)} \colon \art \to \sets$ by
	  \begin{align}
		& \Def_{(S, F)}(A)  &:=&  \{\text{iso. classes of deformations of $(S,F)$ over $A$}\}, \\
		& \Def_{S}(A)   &:=&  \{\text{iso. classes of deformations of $S$ over $A$}\}, \notag \\
	  	& \Def_{F}(A)   &:=&  \{ \text{iso. classes of deformations of $F$ over $A$} \}. \notag
	  \end{align}
\end{defi}

The  automorphism groups $\Aut(S,F)$, $\Aut(S)$, and $\Aut(F)$  act on appropriate
 deformations functors, and this action will be studied in \S\ref{Sec: AutGrps}.
The reader should be familiar with the  definitions of $\Aut(S)$ and $\Aut(F)$, but perhaps not of
 $\Aut(S,F)$.

\begin{defi}\label{Def: Autom}
An \textbf{automorphism} of $(S,F)$ is a pair $(\sigma, \tau)$ that consists of:
\begin{enumerate}
\item  an automorphism $\sigma:  S \isoto S$;
\item an isomorphism of sheaves  $\tau: \sigma_*F \isoto F$.
\end{enumerate}
The group of automorphisms of $(S,F)$, denoted by $\Aut(S,F)$, fits into the exact sequence
\begin{equation}\label{Eqn: SeqAutom}
\begin{aligned}
 0\to \Aut(F) \to \Aut(S,F)& \to \Aut(S) \\
 (\sigma, \tau) &\mapsto \sigma.
\end{aligned}
\end{equation}
\end{defi}

These automorphism groups act naturally on their respective functors.

\begin{defi} \label{Def: DeformAction}
	Let $(S, F)$ be a given pair.  Then we define the \textbf{natural action} of
		\begin{itemize}
			\item  $\Aut(S,F)$ on $\Def_{(S,F)}$ by making an element $(\sigma, \tau) \in \Aut(S, F)$ acts as 
				 $$(S_{A}, F_A, i, , j) \mapsto (S_{A}, F_A, \sigma \circ i, \tau \circ \sigma_*(j)).$$  Here
				 $\tau\circ \sigma_*(j)$ is the composition $\sigma_*i_*(F_A\otimes_A k)\stackrel{\sigma_*(j)}{\longrightarrow}\sigma_*(F)\stackrel{\tau}{\longrightarrow} F;$
			\item  $\Aut(S)$ on $\Def_{S}$ by making an element  $\sigma\in \Aut(S)$ acts as
				$(S_{A}, i) \mapsto (S_{A}, \sigma \circ i);$
			\item  $\Aut(F)$ on $\Def_{F}$ by making an element  $\tau \in \Aut(F)$ acts as
				$(F_{A}, j) \mapsto (F_{A}, \tau \circ j).$
		\end{itemize}
\end{defi}

Later we will relate the above deformation functors to the Quot scheme, so it is convenient to introduce
the deformation functors arising from the Quot scheme.  To avoid irrelevant foundational issues, we only
define the deformation functors associated to nodal curves.
\begin{defi}
	Let $X$ be a nodal curve; $F$ a coherent sheaf on $X$; and $q \colon \calO_{X}^{\oplus r} \twoheadrightarrow F$ a surjection.  A \textbf{deformation of the pair} $(X, q)$ over $A \in
\art$ is a quadruple $(X_{A}, i, q_{A}, j)$
where $q_{A} \colon \calO_{X_A}^{r} \twoheadrightarrow F_{A}$ is a surjection
such that $(X_{A}, F_{A}, i, j)$ is a deformation of $(X,F)$ in the sense of Definition~\ref{Def: DeformDef}.
Furthermore, we require that the isomorphism $j \colon i_{*}(F_{A} \otimes_A k) \isoto F$ respects quotient maps, in the
sense that $q = j \circ i_{*}(q_{A} \otimes 1)$.

Given a second deformation $(X'_{A}, i', q'_{A}, j')$ of $(X,q)$ over $A$,
an \textbf{isomorphism} from $(X_{A}, i, q_{A}, j)$ to $(X'_{A}, i', q'_{A}, j')$
is defined to be a pair $(\phi, \psi)$ consisting of

\begin{enumerate}
	\item an isomorphism  $\phi \colon X_{A} \isoto X'_{A}$ over $A$;  
	\item an isomorphism  $\psi \colon \phi_{*}(F_{A}) \isoto F'_{A}$ of $\calO_{A'}$-modules
such that $\psi\circ \phi_*(q_A)=q_A'$.
\end{enumerate}

A \textbf{deformation of $q$} over $A\in \art$ is defined to be a deformation of $(X,q)$ of the form
$(X \otimes_{k} A, i_{\text{can}}, q_{A}, j)$, where $(X \otimes_{k} A, i_{\text{can}})$
is the trivial deformation.  An \textbf{isomorphism} from one deformation of $q$ to another
is defined to be an isomorphism of the associated deformations of $(X, q)$.
\end{defi}

The deformation functors $\Def_{q}$ and $\Def_{(X,q)}$ are defined in the expected manner.

\begin{defi}
We define functors $\Def_{q}, \Def_{(X,q)} \colon
\art \to \sets$ by
	\begin{align}
		\Def_{(X,q)}(A) 	& :=		\{ \text{iso. classes of deformations of $(X,q)$ over $A$} \}, \\
		\Def_{q}(A) 	& := 	\{ \text{iso. classes of deformations of $q$ over $A$}  \}. \notag
	\end{align}
\end{defi}

To study $\pdbst$, we also need a slight generalization of $\Def_{(X, q)}$.
\begin{defi}
	Suppose that $X$ is a stable curve; $F$ a coherent sheaf; $q \colon \calO_{X}^{\oplus r} \twoheadrightarrow F$ a quotient map; and $p \colon X \hookrightarrow \bbP^N$ is a $10$-canonical embedding.  A \textbf{deformation of the pair $(p, q)$} over
$A\in \art$ is a quadruple $(p_{A}, i, q_{A}, j)$, where $p_{A} \colon X_{A} \hookrightarrow \bbP^{N}_{A}$ is closed embedding and
$(X_{A}, i, q_{A}, j)$ is a deformation of the pair $(X, q)$.   We further require
\begin{itemize}
	\item the line bundles $\calO_{X_{A}}(1)$ and $\omega_{X_{A}/A}^{\otimes 10}$ are isomorphic;
	\item $p_{A} \otimes 1= p \circ i$.
\end{itemize}

	Given a second deformation $(p'_{A}, i', q_{A}', j')$ of $(p, q)$, we define an \textbf{isomorphism} from the first deformation to the second to be
an isomorphism $(\phi,\psi)$ of the associated deformations of $(X, q)$ with the property that
	\begin{displaymath}
		p_{A} = p'_{A} \circ \phi.
	\end{displaymath}
\end{defi}

\begin{defi}
	Define the functor $\Def_{(p, q)} \colon \art \to \sets$ by
	  \begin{displaymath}
		\Def_{(p, q)}(A)  :=  \{\text{iso. classes of deformations of $(p,q)$ over $A$}\}. \\
	  \end{displaymath}
\end{defi}

Note that there are forgetful  transformations $\Def_{q} \to \Def_{F}$ and $\Def_{(p,q)}\to \Def_{(X,q)} \to \Def_{(X,F)}$ that are formally smooth
once $F$ is sufficiently positive (see Lemma~\ref{Lemma: QuotToSh}).

The deformation functors we study are parameterized by complete local $k$-algebras.  
There are several different  ways in which a complete local $k$-algebra can parameterize a deformation functor.

 We say that a functor $\Def \colon \art \to \sets$ is \textbf{pro-representable} if it is isomorphic to the formal spectrum functor 
\begin{equation} \label{Eqn: ProYoneda}
\begin{aligned}
\Spf(R)\colon & \art \to \sets \\
	A & \mapsto \Hom_{\text{loc}}(R, A),
\end{aligned}
\end{equation}
for some complete local $k$-algebra $R$ with residue field $k$.
A pair $(R, \pi)$ consisting of such an algebra $R$ and an isomorphism $\pi \colon \h[R]  \isoto \Def$ is said to be a \textbf{universal deformation ring} for $\Def$. An easy application of Yoneda's lemma shows that if $(R, \pi)$ and $(R', \pi')$ are both universal deformation rings for $\Def$, then there is a canonical isomorphism $R \cong R'$.  An exercise in unraveling definitions shows that the completed local ring of an appropriate Quot scheme is a deformation ring for $\Def_{q}$, and
similarly for $\Def_{(X,q)}$.

The functors $\Def_{F}$ and $\Def_{(X, F)}$ are not always pro-representable, but do satisfy the weaker
condition of admitting a miniversal deformation ring.  Suppose that we are given a pair $(R, \pi)$ consisting of
a complete local $k$-algebra $R$ and a natural transformation $\pi \colon \h[R] \to \Def$.  We say that $(R, \pi)$ is
a \textbf{versal deformation ring} for $\Def$ if $\pi$ is formally smooth.  If $\pi$ has the additional property that it induces an isomorphism on tangent spaces, then we say that $(R, \pi)
$ is a \textbf{miniversal} (or semiuniversal) \textbf{deformation ring}.  One can show that if $(R, \pi)$ and $(R', \pi')$ are both miniversal deformation rings for $\Def$, then $R$ is isomorphic to  $R'$, but in
contrast to the situation for deformation rings, there is no distinguished isomorphism $R \cong R'$.  We now proceed to construct miniversal deformation rings for $\Def_{I}$ and $
\Def_{(X, I)}$.

\subsection{The miniversal deformation rings}
The existence of miniversal deformation rings for $\Def_{I}$ and $\Def_{(X, I)}$ can be deduced from theorems
of Schlessinger, but for later computations, we will want an explicit description of these rings.  We derive such a
description by relating $\Def_{I}$ and $\Def_{(X, I)}$ to the analogous deformation functors associated
to the node $\node$.  We begin by fixing some notation for the node.

\begin{defi}\label{Def: Stan-node}
	The \textbf{standard node} $\node$ is the complete local $k$-algebra $k[[x,y]]/(xy)$.  The \textbf{normalization} of the standard node is denoted $\nodenm$.
\end{defi}
As a subring of the total ring of fractions $\operatorname{Frac}(\node)$, the normalization of $\node$  is equal to $\nodenm = \node[x/(x+y)]$.  It follows that the quotient $\nodenm/
\node$  is a 1-dimensional $k$-vector space spanned by the image of $x/(x+y)$.
Recall that $\tilde{\mathcal O}_0$ is also isomorphic to the ring $k[[x]]\oplus k[[y]]$, and the inclusion $\mathcal O_0\to \tilde {\mathcal O}_0$   factors as
$$
\frac{k[[x,y]]}{(xy)}\to k[[x]]\oplus k[[y]]\isoto \frac{k[[x,y]]}{(xy)}\left[\frac{x}{x+y}\right]
$$
where the first map is given by $h(x,y)\mapsto (h(x,0),h(0,y))$ and the second map is given  by $(f,g)\mapsto (fx+gy)/(x+y)$.

Over $\node$, there are exactly two rank $1$, torsion-free modules up to isomorphism: the free module and a unique
module that fails to be locally free.  A proof of this statement can be found in \cite{dsouza}, where it is deduced from \cite[Thm.~3.1]{vasconcelos}.  There are several ways to describe
the module that fails to be locally free.
\begin{defi} \label{Def: NotFreeMod}
	The unique rank 1, torsion-free module $I_0$ over $\node$ that fails to be locally free can be described as any one of the following modules:
	\begin{enumerate}
		\item the ideal $(x, y) \subset \node$, considered as an $\node$-module,
		\item the extension $\nodenm \supset \node$, considered as an $\node$-module,
		\item the $\node$-module with presentation $\langle e, f \colon y \cdot e = x \cdot f = 0\rangle$.
	\end{enumerate}
\end{defi}
An isomorphism from the 3rd module to the 1st module is given by $e \mapsto x$, $f \mapsto y$, while an isomorphism from the 3rd
to the 2nd is given by $e \mapsto x/(x+y)$, $f \mapsto y/(x+y)$.  In passing from one model of $I_0$ to another, we will always implicitly identify the modules via these specific
isomorphisms.

\subsubsection{Formal smoothness and reduction to the case of nodes}  \label{Subsec: ReduceToNode}
If $I$ is a rank $1$, torsion-free sheaf on a nodal curve $X$, then the study of $\Def_{I}$ and $\Def_{(X,I)}$
reduces to the study of $\Def_{I_0}$ and $\Def_{(\node, I_0)}$.  Indeed, say that $\Sigma$
is the set of nodes where $I$ fails to be locally free.  For a given $e \in \Sigma$, let $X_e$ denote the spectrum of
the completed local ring $\widehat{\calO}_{X,e}$ and $I_e$ the pullback of $I$ to $X_e$.  There are
forgetful transformations relating global deformations to local deformations:
	\begin{align} \label{Eqn: LocToGlo}
		\Def_{(X,I)} 		& \to \prod_{e \in \Sigma} \Def_{(X_{e}, I_{e})}, \\
		\Def_{I} 		& \to \prod_{e \in \Sigma} \Def_{I_{e}},  \nonumber \\
		\Def_{X} 		& \to \prod_{e \in \Sigma} \Def_{X_e}. \nonumber
	\end{align}
	All of these transformations are formally smooth.  Indeed, for the last  transformation, this
	is \cite[Prop.~1.5]{deligne}.  That result together with \cite[A.1-4]{fantechi} shows that the first
	transformation is formally smooth.  Essentially the same argument also shows that the middle
	transformation is formally smooth, and this is a special case of  \cite[B.1]{fantechi}.

We now construct deformation rings for $\Def_{I_0}$ and $\Def_{(\node, I_0)}$.
We begin by parameterizing deformations of $(\node, I_0)$.

\begin{defi} \label{Def: WhatIsS2}
	Define  $S_2 = S_2(\node, I_0):= k[[t, u, v]]/( u v - t).$ 
	The deformation $(\mathcal O_{S_2},  I_{S_2},$ $ i, j)$ of $(\mathcal O_0, I_0)$ over $S_2$ is defined by
	setting
		\begin{itemize}
			\item $\mathcal O_{S_2}:= S_2[[x, y]]/( x y -t);$
			\item $I_{S_2}$ equal to the $\mathcal O_{S_2}$-module with presentation
					\begin{equation}
				I_{S_2}:=
						\langle \tilde{e}, \tilde{f} : y \cdot \tilde{e} = - u \cdot \tilde{f}, x \cdot \tilde{f} = -v \cdot \tilde{e}
						\rangle;
					\end{equation}
			\item $i \colon \mathcal O_{S_2}  \otimes_{S_2} k \isoto \node$ equal to the isomorphism that is the identity on the variables $x$ and $y$;
			\item $j \colon i_{*}(I_{S_2} \otimes_{S_2} k) \isoto I_0$ equal to the isomorphism given by rules $\tilde{e} \otimes 1
					\mapsto e$ and $\tilde{f} \otimes 1 \mapsto f$.
		\end{itemize}
\end{defi}

Deformations of $I_0$ alone are parameterized similarly.
\begin{defi} \label{Def: WhatIsS1}
	Define $S_1 = S_1(I_0) := k[[u, v]]/(u v)$.  The algebraic deformation $(I_{S_1}, j)$ of $I_0$ over $S_1$ is defined by
	setting
	\begin{itemize}
	\item $\mathcal O_{S_1}= S_{1}[[x,y]]/(xy)$;
		\item $I_{S_1}$  equal to the $\calO_{S_1}$- module with presentation
					\begin{equation}
						\mathcal I :=
						\langle \tilde{e}, \tilde{f} : y \cdot \tilde{e} = - u \cdot \tilde{f}, x \cdot \tilde{f} = -v \cdot \tilde{e}
						\rangle;
					\end{equation}
		\item $j \colon i_{*}(I_{S_1} \otimes_{S_1} k) \isoto I_0$ equal to isomorphism given by rules $\tilde{e} \otimes 1
					\mapsto e$ and $\tilde{f} \otimes 1 \mapsto f$.
	\end{itemize}
\end{defi}

\begin{remark}
It may be more intuitive to describe the deformations in geometric terms.  There is a versal deformation (resp. trivial deformation) $\mathscr X\to B$ of the node, with base $$
B=\operatorname{Spec}k[u,v,t]/(uv-t) \ \
(\text{resp.}\ \ \ B=\operatorname{Spec}k[u,v]/(uv))
$$
and total space
$$
\mathscr X=B\times \operatorname{Spec}k[x,y]/(xy-t) \ \  (\text{resp.} \ \  \mathscr X=B\times \operatorname{Spec}k[x,y]/(xy)).
$$
The module
$I_{S_2}$ (resp. $I_{S_1}$) is essentially the ``universal'' ideal $I=(x-u,y-v)\subseteq \Gamma(\mathscr X,\mathscr O_X)$ considered as a module as in Definition~\ref{Def:
NotFreeMod} (3).
\end{remark}

\begin{lemma} \label{Lemma: LocDefRing}
	 $S_2$ is a miniversal deformation ring for $\Def_{(\node, I_0)}$.  More precisely, the algebraic deformation $(\calO_{S_2}, i,  I_{S_2}, j)$ defines a transformation $\h[S_2] \to
\Def_{(\node, I_0)}$ that realizes $S_2$ as the miniversal deformation ring for $\Def_{I}$.   Similarly, $S_1$ is a miniversal deformation ring for $\Def_{I_0}$.
\end{lemma}
\begin{proof}
	The claim concerning the ring $S_1$ was established in the course of proving Proposition~2.6 of \cite{CMK}.  The same argument holds for $S_2$ provided that one replaces
the standard irreducible, nodal plane cubic used in that proof with a general pencil containing such a curve.
\end{proof}

Given a rank $1$, torsion-free sheaf $I$ that fails to be locally free at a set of nodes $\Sigma$, there is a simple relation between
$\Def_{I}$ and $\prod_{e \in \Sigma} \Def_{I_e}$.
\begin{defi}
	Let $\Def^{{\rm l.t.}}_{I} \subset \Def_{I}$ be the subfunctor parameterizing deformations that map to the trivial
	deformation under $\Def_{I} \to \prod_{e \in \Sigma} \Def_{I_e}$.  Define $\Def_{(X,I)}^{{\rm l.t.}}$ similarly.
	Elements of these deformation functors (valued in a given ring) are called \textbf{locally trivial} deformations (over that ring).
\end{defi}

\begin{lemma} \label{Lemma: LocTriv}
	Let $X$ be a nodal curve; $\Sigma$ a set of nodes; $g \colon X_{\Sigma} \to X$ the map that normalizes the nodes $\Sigma$;
and $I := g_{*}(L)$ the direct image of a line bundle $L$ on $X_{\Sigma}$.  Then the rule
\begin{equation} \label{Eqn: LocTrivPara}
\begin{aligned}
     \Def_{L}(A) & \longrightarrow \Def_{I}^{{\rm l.t.}}(A)\\
	(L_A, j) & \mapsto ((g\times \id)_{*}(L), (g\times \id)_{*}(i))
\end{aligned}
\end{equation}
for any $A\in \art$, defines an isomorphism $\Def_{L} \stackrel{\cong}{\longrightarrow} \Def_{I}^{{\rm l.t.}}$.
\end{lemma}
\begin{proof}
	The map $\Def_{L} \to \Def_{I}$ defined by Eqn.~\eqref{Eqn: LocTrivPara} has the property that the composition $\Def_{L} \to \Def_{I} \to \prod_{e \in \Sigma} \Def_{I_e}$ is the
trivial map, so there is an induced map $\Def_{L} \to \Def_{I}^{{\rm l.t.}}$.
Studying the map $\Def_{I} \to \prod \Def_{I_e}$ and the associated map on tangent-obstruction theories, one can show using the local-to-global spectral sequence for $\operatorname{Ext}$ that
$\Def_{I}^{{\rm l.t.}}$ is formally smooth with tangent space $T(\Def_{I}^{{\rm l.t.}}) = H^{1}(\underline{\End}(I))$.  This vector space is just
$H^{1}(X_{\Sigma}, \calO_{X_{\Sigma}})$ (see e.g.~the proof of Lemma \ref{Lemma: aut-I}), which can be  identified with the
tangent space to $\Def_{L}$ in such a way that $T(\Def_{L}) \to T(\Def_{I}^{{\rm l.t.}})$ is the identity.  By formal
smoothness, it follows that $\Def_{L} \to \Def_{I}^{{\rm l.t.}}$ is an isomorphism.
\end{proof}

Let us denote by $R_1$ the miniversal deformation ring of $\Def_I$ and by $R_2$ the miniversal deformation ring of
$\Def_{(X,I)}$ (which exists by, say, \cite[\S~A]{fantechi}).
 Lemma~\ref{Lemma: LocDefRing} together with the discussion following Eqn. \eqref{Eqn: LocToGlo} allows us to describe the miniversal deformation rings $R_1$ and $R_2$ as
follows.

\begin{cor} \label{Cor: RingDesc}
	Let $X$ be a nodal curve; $I$ a rank $1$, torsion-free sheaf on $X$; and $\Sigma$ the set of nodes
	where $I$ fails to be locally free.  For every $e \in \Sigma$,
	fix an identification of $(\hat{\calO}_{X,e}, I_e)$ with $(\node, I_0)$.
	Then the forgetful transformations in Eqn.~\eqref{Eqn: LocToGlo} induce inclusions
	\begin{align*}
		\widehat{\bigotimes_{e \in \Sigma}} k[[ U_{\el}, U_{\er}]]/( U_{\el} U_{\er}) \cong \widehat{\bigotimes}_{e \in \Sigma}  S_{1} &  \hookrightarrow  R_1, \\
		\widehat{\bigotimes_{e \in \Sigma}} k[[U_{\el}, U_{\er}, T_e]]/(U_{\el} U_{\er}-T_e) \cong \widehat{\bigotimes}_{e \in \Sigma}  S_{2}& \hookrightarrow  R_2,
	\end{align*}
	and each inclusion realizes the larger ring as a power series ring over the smaller ring.
\end{cor}

\section{Automorphism Groups and Their Actions} \label{Sec: AutGrps}
Automorphism groups appeared in the previous section, where we defined group actions on deformation functors
(Def.~\ref{Def: DeformAction}).  Here we study the structure of these groups with the aim of collecting results to use in \S\ref{Sec: ActionsOnRings}.  There we will study the problem of
lifting the action of an automorphism group on a
deformation functor to an action on a miniversal deformation ring.  The existence of a lift
follows from a theorem of Rim if  the automorphism group is known to be linearly reductive.  Thus, the focus
of this section is on showing that the automorphism groups of interest are linearly reductive.

We begin by studying automorphisms of the node $\node$ (Def.~\ref{Def: Stan-node}) and its unique rank $1$, torsion-free module $I_0$ that fails to be locally free
(Def.~\ref{Def: NotFreeMod}).  The automorphism group $\Aut(X_0,I_0)$ fits into the exact sequence
		\begin{equation} \label{Eqn: LocExactSeq}
		         \begin{CD}		
		         		0 @>>> \Aut(I_0)			@>>>	\Aut(X_0, I_0)			@>>>	\Aut(X_0) @>>> 0,
			\end{CD}
		\end{equation}
and the group $\Aut(I_0)$ admits the following explicit description.

\begin{lemma} \label{Lemma: GlShAut}
	Consider $I_0$ as the normalization $\tilde{\node}$.
	Then the natural action of $\tilde{\node^{\ast}}$ on $I_0$ induces an isomorphism $\tilde{\node^{\ast}} \isoto \Aut(I_0)$.
\end{lemma}
\begin{proof}
We claim that every $\node$-linear map $\phi \colon I_0 \to I_0$ is $\tilde{\node}$-linear.  It is enough to
show that $\phi$ commutes with multiplication by $x/(x+y)$, and this is clear: for all $s \in I_0$, we have
$$(x+y) \cdot \phi( x/(x+y) \cdot s) =  \phi( x \cdot  s) =  x \cdot  \phi(s).$$
Dividing by $x+y$, we obtain the desired equality.  Thus, $\Aut(I_0)$ coincides with the group of
$\tilde{\node}$-linear automorphisms, which equals $\tilde{\node^{\ast}}$.
\end{proof}

The action of $\tilde{\calO_0^{\ast}}$ can also be described in terms of the presentation from Definition~\ref{Def: NotFreeMod}.
A typical element  $f \in \tilde{\calO_0^{\ast}}$ can be uniquely written as $f = \alpha \frac{x}{x+y} + \beta \frac{y}{x+y} + g(x,y)$,
with $\alpha, \beta \in k^{\ast}$ and $g(x,y) \in (x,y) \subset \node$, and this element acts by
\begin{equation*}
	e \mapsto (\alpha + g(x,0)) e, \hspace{1cm} f \mapsto (\beta + g(0,y)) f.
\end{equation*}

We now turn to the global picture. Let $I$ be a rank $1$, torsion-free sheaf on a nodal curve $X$.  Set $\Sigma$ equal to the
set of nodes where $I$ fails to be locally free.  In analogy with Eqn.~\eqref{Eqn: LocExactSeq}, $\Aut(X,I)$ fits into the following exact sequence:
		\begin{equation}\label{Eqn: GloExactSeq}
			\begin{CD}
				0 @>>>	\Aut(I)			@>>>	\Aut(X, I)			@>>>	\Aut(X) .
			\end{CD}	
		\end{equation}
We describe $\Aut(X, I)$ by describing the outermost groups.

Consider first $\Aut(X)$.  Without more information, we can only describe the rough features of this group.  For $X$ stable (the main case of interest), $\Aut(X)$ is a finite, reduced
group scheme (\cite[Thm.~1.11]{deligne}), and if we additionally assume that $X$ is general and of genus $g \ge 3$, then this group is trivial.  However, $\Aut(X)$ can be highly non-
trivial for special curves: see \cite{vOV} for a sharp bound on the cardinality of $\Aut(X)$ in terms of the genus $g$, and for a description of the curves attaining the bounds.

The group $\Aut(I)$ admits the following explicit description. In the notation from \S\ref{Subsec: ReduceToNode},  there is a natural map
$\Aut(I) \to \Aut(I_e)$ for every $e \in \Sigma$, and we use this map to describe $\Aut(I)$.

\begin{lemma}\label{Lemma: aut-I}
	Let $X$ be a nodal curve; $I$ a rank $1$, torsion-free sheaf; $\Sigma$ the set of points where $I$
fails to be locally free; and $g \colon X_{\Sigma} \to X$ the map that normalizes the nodes $\Sigma$.
	Then there is a unique isomorphism $H^{0}(X_{\Sigma}, \calO_{X_{\Sigma}}^{\ast}) \cong \Aut(I)$ that extends the inclusion of $H^{0}(X, \calO_{X}^\ast)$ in $\Aut(I)$ and makes
the
	diagram
	\begin{displaymath}
		\xymatrix{
			H^{0}(X_{\Sigma}, \calO_{X_{\Sigma}}^{\ast}) 	\ar[r]^{\cong} \ar[d]&  	\Aut(I) \ar[d]\\
			\tilde{\calO}_{X,e}^{\ast}	\ar[r]^{\cong}	&			\Aut(I_{e})
		}
	\end{displaymath}
	commute for all $e \in \Sigma$.  Here $\tilde{\calO}_{X,e}$ is the normalization of the completed local ring at $e$, the horizontal maps are isomorphisms, and  the vertical maps
are restrictions.
\end{lemma}
\begin{proof}
Given $I$, we prove the stronger statement that
$\underline{\End}(I)$ is canonically isomorphic to $g_{*}(\calO_{X_{\Sigma}})$.  Because $I$ is torsion-free, $\underline{\End}(I)$ injects into $\underline{\End}(I \otimes
\operatorname{Frac}(\calO_{X}))$, which equals $\operatorname{Frac}(\calO_{X})$ as $I$ is rank $1$.  Thus, $\underline{\End}(I)$  is a finitely generated, commutative $\calO_X$-
algebra satisfying $\calO_{X} \subset \underline{\End}(I) \subset \operatorname{Frac}(\calO_{X})$.   Furthermore, an application of the Cayley-Hamilton Theorem shows that a  local
section of  $\underline{\End}(I)$ satisfies a monic equation whose coefficients are local sections of $\calO_{X}$.  We  may conclude that $\underline{\End}(I) \subset \nu_{*}
(\calO_{\tilde{X}})$, where $\nu \colon \tilde{X} \to X$ is the (full)  normalization. To complete the proof, it is enough to show that the support of $\nu_{*}(\calO_{\tilde{X}})/\underline{\End}
(I)$ is precisely   $\Sigma$.  However, this can be checked on the level of completed stalks, and so we may deduce the claim from the Lemma~\ref{Lemma: GlShAut}.   The result now
follows by taking global sections of $\underline{\End}(I)$ and passing to units.
\end{proof}

One consequence of the previous two lemmas is that many of the groups appearing in this paper are linearly reductive.
Recall that the ground field $k$ may have positive characteristic, and in positive characteristic linear reductivity is a strong condition to impose.  Indeed, while many algebraic groups
(e.g.~$\operatorname{GL}_{r}, \operatorname{SL}_{r}, \dots$) are linearly reductive in characteristic $0$, Nagata has shown that the only linearly reductive groups
in characteristic $p>0$ are  the groups $G$ whose identity component $G_0$ is a multiplicative torus
and whose \'{e}tale quotient $G/G_0$ has prime-to-$p$ order.  We now list the groups we have shown satisfy this
condition.
 
\begin{cor} \label{Cor: GrpRLinRed}
	Let $X$ be a nodal curve and $I$  a rank $1$, torsion-free sheaf.  Then the following groups are reduced and linearly reductive:
	\begin{itemize}
		\item the automorphism group $\Aut(I)$;
		\item the quotient group $\Aut(I_0)/(1+(x,y)\node)$;
		\item the automorphism group $\Aut(X,I)$ when $X$ is stable and does not admit an order $p={\rm char}(k)$ automorphism.	
	\end{itemize}
\end{cor}
\begin{proof}
Lemma~\ref{Lemma: GlShAut} shows $\Aut(I_0)/(1+(x,y) \node)$ is a multiplicative torus, and
Lemma~\ref{Lemma: aut-I} shows the same is true for $\Aut(I)$.  Given this, an inspection of Eqn.~\eqref{Eqn: GloExactSeq} proves that $\Aut(X,I)$ is linearly reductive.
\end{proof}

\section{Group Actions on Rings} \label{Sec: ActionsOnRings}
In this section we show that, in the cases of interest, the actions on deformation functors from Definition~\ref{Def: DeformAction}
lift to unique actions on miniversal deformation rings (Fact~\ref{Fact: Rim}), which we then compute (Thm.~\ref{Thm: MainActionThm}). These results are used in \S\ref{Sec: Luna},
where we show that the action on the miniversal deformation
ring can be described using the GIT construction of the compactified Jacobian (Lemma~\ref{Lemma: SliceToRing},
Lemma~\ref{Lemma: StabToAut}).  We then use this observation to deduce the main theorem of the paper
(Thm.~\ref{Thm: MainLocStr}).  Key to this section are the linear reductivity results from the previous section.

We begin by showing that certain actions are trivial.
\begin{lemma}\label{lem61}
	The action of $1 + (x,y) \node \subset \Aut(I_0)$ on $\Def_{I_0}$ is trivial.
\end{lemma}

\begin{proof}
	Suppose we are given  $A \in\art$ and a deformation $(I_A, j)$ of $I_0$ over $A$.  Given $\tau \in 1 + (x, y) \node$, we must show that $(I_{A}, j)$ and $(I_{A}, \tau^{-1} \circ j)$
are isomorphic deformations.  But this is clear: $\tau$ lies in $\node$, and multiplication by
$\tau \otimes 1 \in \node \otimes_{k} A$ defines an isomorphism $(I_{A}, j) \isoto (I_{A}, \tau^{-1} \circ j)$.
\end{proof}
Essentially the same argument proves the following two lemmas.

\begin{lemma} \label{Lemma: ScalarsRTrivial}
	Let $X$ be a nodal curve and $I$ a rank $1$, torsion-free sheaf on $X$.  Then the subgroup $\bbG_m \subset \Aut(I)$ of scalar automorphisms
	acts trivially on $\Def_{I}$.  Under the inclusion \eqref{Eqn: GloExactSeq}, $\bbG_m$ also acts trivially on $\Def_{(X,I)}$.
\end{lemma}

\begin{proof}
We give a proof for $\Def_I$; the case of $\Def_{(X,I)}$ is similar, and left to the reader.  	If $(I_{A}, j)$ is a deformation of $I$ and
	$\tau \in \bbG_m \subset \Aut(I)$ a scalar automorphism, then $\tau$ trivially extends to
	an automorphism $\tilde{\tau}$ of $I_{A}$ that defines an isomorphism of $(I_{A}, j)$ with
	$(I_{A}, \tau^{-1} \circ j)$.
\end{proof}

\begin{lemma} \label{Lemma: LocTrivRTriv}
	Let $X$ be a nodal curve and $I$ a rank $1$, torsion-free sheaf.  Then $\Aut(I)$ acts trivially on the
	subfunctor $\Def_{I}^{{\rm l.t.}} \subset \Def_{I}$. Under the inclusion \eqref{Eqn: GloExactSeq}, $\Aut(I)$ also acts trivially on the subfunctor  $\Def^{{\rm l.t.}}_{(X,I)}\subset
\Def_{(X,I)}$.
\end{lemma}

\begin{proof}
The lemma is a consequence of Lemmas~\ref{Lemma: GlShAut} and \ref{Lemma: LocTriv}.
\end{proof}

We may now invoke a theorem of Rim to show that the actions uniquely lift to actions
on miniversal deformation rings.

\begin{fact}[Rim \cite{Rim}] \label{Fact: Rim}
	Let $X$ be a nodal curve and $I$ a rank $1$, torsion-free sheaf.  Then:
	\begin{enumerate}[(i)]
	\item  there is a unique action of $\Aut(I_0)$ on the miniversal
	 deformation ring $S_1$ (resp. $S_2$)
	that makes the map $\h[S_1] \to  \Def_{I_0}$ (resp. $\h[S_2] \to  \Def_{(\node,I_0)}$) 	 equivariant and has the property that the subgroup $1 + (x,y)\node \subset
\tilde{\node^{\ast}} = \Aut(I_0)$ acts
trivially;
	\item there is a unique action of $\Aut(I)$ on the miniversal deformation ring $R_1$ of $\Def_I$ that makes $\h[R_1] \to \Def_{I}$ equivariant;
	\item there is a unique action of $\Aut(X,I)$ on the miniversal deformation ring $R_2$ of $\Def_{(X,I)}$ that makes
$\h[R_2] \to \Def_{(X,I)}$ equivariant, provided that $X$ is stable and it does not admit an order $p={\rm char}(k)$ automorphism.
	\end{enumerate}
\end{fact}

\begin{proof}
	This is a special case of \cite[p.~225]{Rim}.  Indeed, the functors $\Def_{I_0}$, $\Def_{I}$, $\Def_{(\node, I_0)}$ and
	$\Def_{(X, I)}$ are all examples of a deformation functor $F$ associated to a ``homogeneous fibered category in groupoid" satisfying a finiteness condition.  Given an action of a
linearly reductive group on such
	a category, there is an induced action on $F$, and Rim's Theorem asserts that
there exists a miniversal deformation ring $R$ that admits an action of $G$ making  $\h[R] \to F$ equivariant.  Furthermore,
as an algebra with $G$-action, $R$ is unique up to a (non-unique) isomorphism.

One may verify that the actions on $\Def_{I_0}$, $\Def_{I}$, $\Def_{(\node, I_0)}$ and $\Def_{(X,I)}$ are defined on the level of groupoids.  The claims concerning $R_1$
and $R_2$ follows immediately  because we have shown that $\Aut(I)$ and $\Aut(X,I)$ are linearly reductive.  The group
$\Aut(I_0)$ is certainly not linearly reductive, but Lemma~\ref{lem61} asserts that this group acts through its linearly reductive quotient $\Aut(I_0)/(1+(x,y)\node)$.  Case (i) then follows
as well.
\end{proof}

The actions described by the Fact~\ref{Fact: Rim} are, of course, unique only up to a non-unique isomorphism.    Because of the non-uniqueness, it is not immediate that the
group action is functorial.  This issue is addressed in the lemma below.

\begin{lemma} \label{Lemma: FunctorialAction}
	Let $X$ be a nodal curve and $I$ a rank $1$, torsion-free sheaf.  For every point $e \in X$ where
$I$ fails to be locally free, fix an isomorphism between $(\widehat{\calO}_{X,e}, I \otimes \widehat{\calO}_{X,e})$ and $(\node, I_0)$.
	Then the restriction transformations
\begin{equation}
	\Def_{I} \to \prod_{e \in \Sigma} \Def_{I_0}  \hspace{0.5cm} \text{ resp. } \hspace{0.5cm} \Def_{(X,I)} \to \prod_{e \in \Sigma} \Def_{(\node, I_0)}   
\end{equation}
lift to transformations of miniversal deformation rings
\begin{equation} \label{Eqn: FunActLocGlo}
	\h[R_1] \to \prod_{e \in \Sigma} \h[S_1] \hspace{0.5cm} \text{ resp. } \hspace{0.5cm}  \h[R_2] \to \prod_{e \in \Sigma} \h[S_2] \notag
\end{equation}
that are equivariant with respect to the homomorphism
\begin{equation}
	\Aut(I)\to \prod_{e \in \Sigma}\Aut(I_0)
\end{equation}
and the actions of $\Aut(I)$ and $\Aut(I_0)$ described in Fact~\ref{Fact: Rim}.
\end{lemma}
\begin{proof}
	The only condition that is not immediate is that the natural transformations can be chosen to be equivariant.  We give the proof for $\h[R_1]$ and leave the task of extending the
argument to $\h[R_2]$ to the interested reader.

As $\h[S_1] \to \Def_{I_0}$ is formally smooth, there exists a lift $\h[R_1] \to \prod \h[S_1]$
of the forgetful transformation $\Def_{I} \to \prod \Def_{I_0}$, and such a lift is automatically
formally smooth.  Writing $R_1$ as a power series ring over $\widehat{\otimes} S_1$,
it is easy to see that there exists an action of $\Aut(I)$ on $\h[R_1]$ that makes $\h[R_1] \to \prod \h[S_1]$ equivariant and has the property that the induced action on the tangent
space $T(\h[R_1])$ coincides
with the natural action on $T(\Def_{I})$.
To complete the proof, we must show that this action makes $\h[R_1] \to \Def_{I}$
equivariant, and hence satisfies the conditions of Fact~\ref{Fact: Rim}.

Consider the composition $\h[R_1] \to \prod \Spf(S_1) \to \prod \Def_{I_0}$.  This transformation is formally smooth and hence realizes
$R_1$ as a (non-minimal) versal deformation ring for $\Def_{I_0}$.  Furthermore, the constructed action of
$\Aut(I)$ on $R_1$ makes $\h[R_1] \to \prod \Def_{I_0}$ equivariant and
induces the standard action on $\operatorname{T}(R_1) = \operatorname{T}(\Def_{I})$.    A
second action on $R_1$ with this property is the unique action that makes $\h[R_1] \to \Def_{I}$
equivariant.  An inspection of Rim's \emph{proof} shows that the uniqueness statement in Fact~\ref{Fact: Rim}
still holds if the miniversality hypothesis is weakened to versality, provided the action on the tangent space
is specified.  In particular, there is an automorphism of $R_1$ transforming the first action into the
second.  We can conclude that the map in \eqref{Eqn: FunActLocGlo} and the action in Fact~\ref{Fact: Rim}
can be chosen so that $\h[S_1] \to \Def_{I}$ is equivariant.  This completes the proof.
\end{proof}

We now compute the actions described by Fact~\ref{Fact: Rim}. Let us start with the action of  $\Aut(I_0)$ on $S_1$.

\begin{lemma}	
In terms of the presentation from Definitions~\ref{Def: WhatIsS1}, \ref{Def: WhatIsS2}, define an action of $\Aut(I_0)$ on $S_1$ and $S_2$ by
	making $\tau = a \frac{x}{x+y} + b \frac{y}{x+y} + g \in \Aut(I_0)$ act as
	\begin{equation*}
	u \mapsto a b^{-1} \cdot u, \hspace{1cm} v \mapsto a^{-1} b \cdot v, \hspace{1cm} t \mapsto t.
	\end{equation*}
	Here $a, b \in k^{\ast}$ and $g \in (x,y) \tilde{\node}$.  Then this action is the unique action described by Fact~\ref{Fact: Rim} (i).
\end{lemma}

\begin{proof}  We give a proof for the case of $S_1$; the case of $S_2$ is similar, and left to the reader.
	The rule above is easily seen to define an action of $\Aut(I_0)$ on $S_1$ with the property that $1 + (x,u) \node$ acts
trivially, so we need only show that this action makes $\h[S_1] \to \Def_{I_0}$ into an equivariant map.  In fact, it is enough
to verify this for the subgroup of $\Aut(I_0)$ that consists of elements of the form $\tau := a \frac{x}{x+y} + b \frac{y}{x+y}$
because this subgroup maps isomorphically onto $\Aut(I_0)/( 1+(x,y)\node )$.

Given such a $\tau$, what is the pullback of the miniversal deformation $(I_{S_1}, i)$ under $\tau$?  It is the module with presentation
	\begin{equation} \label{Eqn: TwistedPres}
	 	 \langle \tilde{e}', \tilde{f}' \colon y \cdot \tilde{e}' = - a^{-1} b u \cdot \tilde{f}', x \cdot \tilde{f}' = -a b^{-1} v \cdot \tilde{e}' \rangle,
	\end{equation}
together with the identification $j$ sending $\tilde{e}' \mapsto e$, $\tilde{f}' \mapsto f$.  One isomorphism between this deformation and the deformation $(I_{S_1}, \tau^{-1} \circ j)$ is
\begin{equation*}
	\tilde{e}' \mapsto b^{-1} \tilde{e}, \hspace{1cm} \tilde{f}' \mapsto a^{-1} \tilde{f}.
\end{equation*}
This completes the proof.
\end{proof}

We now turn our attention to the action of $\Aut(I)$ on $R_1$.  It is convenient to introduce some combinatorial language.
\begin{defi}
	Let $e \in \Sigma$ be a node that lies on the intersection of the irreducible components $v$ and $w$.
	Write $\el$ for the pair $(v,w)$ and $\er$ for the pair $(w,v)$.  Define $s, t \colon \{ \er, \el \} \to
	\{ v, w \}$ to be projection onto the first component and onto the second component respectively.
\end{defi}
This notation is intended to be suggestive of graph theory.  We may consider $v$ and $w$ as being
vertices of the dual graph $\Gamma_{X}$ that are connected by an edge corresponding to $e$.
The pairs $\el$ and $\er$ should be thought of as orientations of this edge, and the
maps $s$ and $t$ are the ``source" and ``target" maps sending an oriented
edge to its source vertex and its target vertex respectively.  The relation with graph theory
is developed more systematically by the authors in  \cite{local2}.

The group $\Aut(I)$ can also be described using similar notation.
\begin{defi} \label{Def: TTorus}
	Let $X$ be a nodal curve, $I$ a rank $1$, torsion-free sheaf, $\Sigma$ the set of nodes where $I$ fails to be
locally free, and $V$ the set of irreducible components of $X$.  Define $T_{\Sigma}$ to be the subgroup
\begin{displaymath}
	T_{\Sigma} \subset \prod_{v \in V} \bbG_m
\end{displaymath}
that consists of sequences $(\lambda_{v})$ with the property that $\lambda_{v_1} = \lambda_{v_2}$ for every two components $v_1$ and $v_2$ whose intersection contains some
node \emph{not} in $\Sigma$.
\end{defi}

\begin{remark}\label{Rem: TTorus}
The torus $T_\Sigma$ is isomorphic to $\Aut(I) = H^{0}(X_{\Sigma}, \calO_{X_{\Sigma}}^{\ast})$ (Lemma~\ref{Lemma: aut-I}). Indeed, the element
$\lambda = ( \lambda_v ) \in T_{\Sigma}$ corresponds to the regular function  $f \in H^{0}(X_{\Sigma}, \calO_{X_{\Sigma}}^{\ast})$ that is equal to the constant $\lambda_v$ on the
component $v$.
It is convenient to have the following explicit isomorphism of $\Aut(I)$ with a split torus.      Let $\Gamma_X$ be the dual graph of $X$ and let  $\Gamma=\Gamma_X(\Sigma)$ be the
dual graph of a curve obtained from $X$ by smoothing the nodes not in $\Sigma$.   There is a map  of vertices $c:V(\Gamma_X)\to V(\Gamma)$ (\cite[\S 2.1]{local2}) and it is easy to
check there is an  isomorphism
$$
\phi:T_\Gamma :=\prod_{v\in V(\Gamma)}\mathbb G_m\isoto T_\Sigma =\Aut(I)\subseteq \prod_{w\in V(\Gamma_X)}\mathbb G_m
$$
defined as follows.  Given $(g_v)\in \prod_{v\in V(\Gamma)}\mathbb G_m$, set $\phi((g_v))_w=g_{c(w)}$ for each $w\in V(\Gamma_X)$.
\end{remark}

We use the description of $\Aut(I)$ as $T_{\Sigma}$ to describe the action of $\Aut(I)$ on $R_1$ and on $R_2$.

\begin{thm} \label{Thm: MainActionThm}
	Let $X$ be a nodal curve; $I$ a rank $1$, torsion-free sheaf;
	$\Sigma$  the set of nodes where $I$ fails to be locally free; and $g_\Sigma :=  h^{1}(X_{\Sigma}, \calO_{X_{\Sigma}})$
	the arithmetic genus of $X_{\Sigma}$.  Then:
	\begin{enumerate}[(i)]	
		\item Define an action of $	T_{\Sigma} = \Aut(I)$  on
			\begin{displaymath}
				R_1(\Sigma) := k[[ \{ U_{\el}, U_{\er} \colon e \in \Sigma \}; W_1, \dots, W_{g_\Sigma}]]/(U_{\el} U_{\er}
					\colon e \in \Sigma)
			\end{displaymath}
			 by making $\lambda \in T_{\Sigma}$ act as
			 \begin{equation}\label{eqnaction}
			U_{\er} 	\longmapsto  	\lambda_{s(\er)}  \cdot U_{\er} \cdot \lambda_{t(\er)}^{-1}, \hspace{0.8cm}
			U_{\el} 	\longmapsto 		\lambda_{s(\el)}  \cdot U_{\el} \cdot \lambda^{-1}_{t(\el)}, \hspace{0.8cm}
			W_i 		\longmapsto 		 W_i.		
			\end{equation}	 
			Then there exists an isomorphism $R_1 \cong R_1(\Sigma)$ that identifies the above action of $T_{\Sigma}$ on
			$R_1(\Sigma)$ with the action of $\Aut(I)$ on $R_1$ from Fact~\ref{Fact: Rim}.

\item Suppose $\operatorname{Aut}(X)$ is trivial, and define an action of
$T_{\Sigma}=\operatorname{Aut}(X,I)$
			on

\begin{displaymath}
R_2(\Sigma) := k[[ \{ U_{\el}, U_{\er},T_e \colon e \in \Sigma \}; W_1, \dots, W_{m}]]/(U_{\el} U_{\er}-T_e \colon e \in \Sigma)
\end{displaymath}
for some $m\in \mathbb Z_{\ge 0}$ by making $\lambda \in T_{\Sigma}$ act as in \eqref{eqnaction}  and as $T_e\longmapsto T_e$.
			Then there exists an isomorphism $R_2 \cong R_2(\Sigma)$ that identifies the above action of
			$T(\Sigma)$ on $R_2(\Sigma)$ with the action of $\Aut(X,I)$ on $R_2$ from Fact~\ref{Fact: Rim}.
		\end{enumerate}
\end{thm}

\begin{remark}\label{remintro}
Let $\Gamma=\Gamma_X(\Sigma)$ be the dual graph of any curve obtained from $X$ by smoothing the nodes not in $\Sigma$.
Then one can check that in the notation of the theorem above, $g_\Sigma=g(X)-b_1(\Gamma)$.  It is also easy to see that the action of $T_\Gamma$ on $R_I$ and $R_{(X,I)}$
defined in Theorem~\ref{Thm: MainThmA} agrees with the action of $T_\Sigma$ defined above.
\end{remark}

\begin{proof}
	This is a consequence of results already proven in his section.  We only prove the statement
	about $R_1$ and leave the task of extending the proof to $R_2$ to the interested reader.
	
	Suggestively set
\begin{equation}
	S(\Sigma) := k[[U_{\el}, U_{\er} \colon e \in \Sigma]]/(U_{\el} U_{\er} \colon e \in \Sigma).
\end{equation}
This is a miniversal deformation ring for $\prod \Def_{I_0}$, where the product runs over
the elements of $\Sigma$. If we fix an isomorphism between
$(\widehat{\calO}_{X,e}, I \otimes \widehat{\calO}_{X,e})$ and $(\node, I_0)$ for every node $e \in \Sigma$,
then by Corollary~\ref{Cor: RingDesc} and Lemma~\ref{Lemma: FunctorialAction}, there exists
an equivariant map $S(\Sigma) \hookrightarrow R_1$ realizing
$R_1$ as a power series ring over $S(\Sigma)$.  To complete the proof,
we need to show that there exists an expression of $R_1$ as a power series
ring generated by variables invariant under the group action.

Thus, consider the map from the cotangent space of
$\h[R_1]$ to the cotangent space of $\Def^{{\rm l.t.}}_{I}$.
This is an equivariant map, and the action of $\Aut(I)$ on the
target space is trivial (Lemma~\ref{Lemma: LocTrivRTriv}).
Because $\Aut(I)$ is linearly reductive, we can find invariant
elements $\bar{W}_1, \dots, \bar{W}_{g_\Sigma} \in R_1$ whose images in the cotangent space $\mathfrak{m}/\mathfrak{m}^2$
map isomorphically onto the cotangent space of $\Def^{{\rm l.t.}}_{I}$.

Letting $W_1, \dots, W_{g_\Sigma}$ denote indeterminates, define
a map
\begin{displaymath}
 	\phi \colon S(\Sigma)[[W_1, \dots, W_{g_\Sigma}]] \to R_1
\end{displaymath}
by sending $W_i$ to $\bar{W}_i$.  The target and source
of $\phi$ are isomorphic, and the induced map on tangent spaces is an isomorphism,
hence $\phi$ itself must be an isomorphism.

Furthermore, if we make $T_{\Sigma}$ act on $S(\Sigma)[[W_1, \dots, W_{g_\Sigma}]]$
by making the group act trivially on the indeterminates, then $\phi$ is
equivariant.   The ring $S(\Sigma)[[W_1, \dots, W_{g_\Sigma}]]$, together with this
group action, is nothing other than $R_1(\Sigma)$, so the proof is complete.
\end{proof}

Observe that the theorem computes the action of $\Aut(X,I)$ on $R_2$ when $X$ is automorphism-free.
It would interesting to compute the action when $X$ is stable, but possibly admits non-trivial
automorphisms.  Indeed, such a result (combined with a suitable extension of
Theorem~\ref{Thm: MainLocStr}) would allow us to remove the hypothesis that $X$
does not have an automorphism from Theorem~\ref{Thm: MainThmA}.   When
$X$ does not admit an automorphism of order $p={\rm char}(k)$, Fact~\ref{Fact: Rim} states that there
is a unique action of $\Aut(X,I)$, so the problem is to modify the action described in Theorem~\ref{Thm: MainActionThm}
to incorporate $\Aut(X)$.
The case where $X$ admits an order $p={\rm char}(k)$ automorphism is more challenging for then we can no longer
cite Rim's work to assert that $\Aut(X,I)$ acts on $R_2$ or to assert that such an action, if it exists,
is unique.  Simply knowing if $R_2$ still admits an unique action of $\Aut(X,I)$ would be interesting.  More generally, it would be interesting to know if Rim's Theorem remains true if
the assumption that the group $G$ acting is linearly reductive is weakened.

\section{Luna Slice Argument}\label{Sec: Luna}
We now prove that the invariant subrings in Theorem~\ref{Thm: MainActionThm} are isomorphic to the
completed local rings of the compactified Jacobians.  The main result is the following.

\begin{thm} \label{Thm: MainLocStr}
	Let $X$ be a nodal curve and  $I$ a rank $1$, torsion-free sheaf. 
	\begin{enumerate}[(i)]	
		\item   \label{Thm: MainLocStr1} Let $\bar{J}(X)$ be a compactified Jacobian of $X$ and assume that $I$ is polystable with respect to the associated stability condition.
		Then the action from Fact~\ref{Fact: Rim} of $\Aut(I)$ on the deformation ring $R_1$ parameterizing deformations of
		$I$ satisfies 	
		\begin{displaymath}
			\widehat{\calO}_{\bar{J}(X), [I]} \cong R_1^{\Aut(I)}.
		\end{displaymath}

		\item  \label{Thm: MainLocStr2} Assume  $X$ is stable and does not admit an order $p$ automorphism, and $I$ is slope polystable
		with respect to the dualizing sheaf $\omega_X$.
		Then the action of $\Aut(X,I)$ on the deformation ring $R_2$ satisfies
	\begin{displaymath}	
					\widehat{\calO}_{\pdbst,[(X,I)]}  \cong  R_2^{\Aut(X,I)}.
	\end{displaymath}
	\end{enumerate}
\end{thm}

	In the theorem, the isomorphisms between the complete local rings are non-canonical, but this is necessarily so as the rings $R_1$ and $R_2$ are themselves only defined up
to non-canonical isomorphism.
	
\begin{remark}\label{remmtpf}
Observe that Theorem~\ref{Thm: MainLocStr}, together with Theorem~\ref{Thm: MainActionThm}, establishes Theorem~\ref{Thm: MainThmA} (see also Remarks~\ref{Rem:
TTorus}, \ref{remintro}).    An elementary argument in GIT shows that the ring $\widehat {B(\Gamma)}^{T_\Gamma}$ defined in Theorem~\ref{Thm: MainThmA} has dimension
$b_1(\Gamma)+\#E(\Gamma)$.  Since $\bar J_{d,g}$ has dimension $4g-3$, it follows that $m=4g-3-b_1(\Gamma)-\# E(\Gamma)$ in Theorem~\ref{Thm: MainActionThm}.
\end{remark}

 The proof of Theorem~\ref{Thm: MainLocStr} is given at the end of the section, where it is deduced from the following sequence of lemmas.

\begin{lemma} \label{Lemma: QuotToSh}
	Let $X$ be a nodal curve; $I$ a rank $1$, torsion-free sheaf; and
	$q \colon \calO_{X}^{\oplus r} \twoheadrightarrow I$ a surjection.
	If $H^{1}(X, I)=0$, then the forgetful morphism $\Def_{q} \to \Def_{I}$
	is formally smooth.
	Assume further that $X$ is stable and $p \colon X \hookrightarrow \bbP^N$ is a
	$10$-canonical embedding.  Then $\Def_{(p,q)} \to \Def_{(X,I)}$ is formally smooth.
\end{lemma}
\begin{proof}
	We prove the statement about $\Def_{q} \to \Def_{I}$ and leave the proof for $\Def_{(p,q)} \to \Def_{(X, I)}$ to the interested
	reader.  Given a surjection $B \twoheadrightarrow A$ of artin local $k$-algebras,
	a deformation $(I_B, j)$ of $I$ over $B$, and a deformation $(q_{A}, j)$ of $q$
	such that the associated deformation of $I$ is isomorphic to $(I_{B} \otimes_{B} A, j \otimes 1)$,
	we must show that there exists a deformation $(q_{B}, j)$ extending $(q_A, j)$ and inducing $(I_B, j)$.
	A filtering argument shows that the vanishing $H^{1}(X,I)=0$ implies that $H^{0}(X_{B}, I_{B}) \to H^{0}(X_{A}, I_{A})$
	is surjective.  Now suppose $s_1, \dots, s_r \in H^{0}(I_{A})$ is the image of the standard basis for
	$H^{0}(X_{A}, \calO^{\oplus r}_{X_A})$.  If we lift these elements to $\tilde{s}_1, \dots, \tilde{s}_r \in H^{0}(X_{B}, I_{B})$
	and define $q_{B} \colon \calO^{\oplus r}_{X_B} \twoheadrightarrow I_{B}$ to be the map the sends the $i$-th standard basis element
	to $\tilde{s}_i$, then $(q_B, j)$ has the desired properties.
\end{proof}

We now relate $R_1$ and $R_2$ to the appropriate Quot schemes.

\begin{lemma} \label{Lemma: SliceToRing}
	Let $X$ be a nodal curve $X$; $I$ a rank $1$, torsion-free sheaf $I$; $q \colon \calO^{\oplus r}_{X} \twoheadrightarrow I$ a quotient map corresponding to a point $\tilde{x}\in \operatorname{Quot}(\calO_{X}^{\oplus r})$.
	 Assume:
	\begin{itemize}
		\item $H^{1}(X, I)=0$;
		\item $q \colon H^{0}(X, \calO_{X}^{\oplus r}) \to H^{0}(X, I)$ is an isomorphism.
	\end{itemize}
		\begin{enumerate}[(i)]	
		\item \label{Lemma: SliceToRing1} If $Z$ is a slice through $\tilde{x}$  in some invariant affine open neighborhood  $\tilde{x}\in U\subseteq \Quot( \calO_{X}^{\oplus r})$,
		 then the completed local ring $\widehat{\calO}_{Z, \tilde{x}}$ of $Z$ at  $\tilde{x}$
			is a miniversal deformation ring for $\Def_{I}$.

		 \item	\label{Lemma: SliceToRing2}  
		\noindent Assume additionally that $X$ is stable.  Let $p \colon X \hookrightarrow \bbP^N$ be a  $10$-canonical 	embedding with $(p,q)$ corresponding to the point  $\tilde{y}$ of
			the relative Quot scheme $\operatorname{Quot}( \calO_{X_g}^{\oplus r})$ (as in the proof of Fact \ref{F:univJac}). 			
                 If $Z$ is a slice through $\tilde{y}$ in some invariant affine open neighborhood $\tilde{y}\in V \subseteq \Quot(\calO^{\oplus r}_{X_g})$, 
                 then  the completed local ring of $\widehat{\calO}_{Z, \tilde{y}}$ of $Z$ at  $\tilde{y}$ is a miniversal deformation ring  for $\Def_{(X,I)}$.

	         \end{enumerate}
\end{lemma}
\begin{proof}
	We prove the statement relating $\Quot(\calO_{X}^{\oplus r})$ to $\Def_{I}$ and leave the
	task of extending the argument to $\Def_{(X, I)}$
	to the interested reader.  The necessary changes are primarily notational  (e.g.~the
	action of $\operatorname{SL}_{r}$ must be
	replaced with that of $\operatorname{SL}_{r}
	\times \operatorname{SL}_{N+1}$).

	Temporarily set $F$ equal to the functor pro-represented by $\widehat{\calO}_{Z, x}$.
	There is a natural forgetful map $\Def_{q} \to \Def_{I}$, and our goal is to show that the
	restriction of this map to $F$ is formally smooth and an isomorphism on tangent spaces.
	We do this by proving that $F(A) \to \Def_{I}(A)$ is injective for $A = k[\epsilon]$ and has the same image
	as $\Def_{q}(A) \to \Def_{I}(A)$ for all  $A \in \art$.  Because
	$\Def_{q} \to \Def_{I}$ is formally smooth (Lemma~\ref{Lemma: QuotToSh}), the lemma will then follow.
	
	The desired facts are proven by studying the action of the lie algebra of $\operatorname{SL}_{r}$
	on deformations.  Set $\mathfrak{sl}_{r}$ equal to the deformation functor
	pro-represented by the completed local ring of $\operatorname{SL}_{r}$ at the
	identity and $\mathfrak{h}$ equal to the deformation functor
	associated to the stabilizer $H := \operatorname{Stab}(\tilde{x}) \subset \operatorname{SL}_{r}$.
	There is a natural map $\mathfrak{sl}_{r}/\mathfrak{h} \to \Def_{q}$
	given by the derivative of the orbit map.  Concretely, this is defined
	by the rule $g \mapsto g \cdot v_{\text{triv}}$, where $v_{\text{triv}}$
	is the trivial deformation (over an unspecified artin local algebra).  Because
	$U$ admits a slice, there exists a morphism $\Def_{q} \to \mathfrak{sl}_{r}/\mathfrak{h}$
	that is a contraction onto the orbit in the sense that the derivative of the orbit map defines a section.  Furthermore,
	this morphism has the property that the preimage of the trivial element $0 \in \mathfrak{sl}_{r}/\mathfrak{h}(A)$
	is $F(A) \subset \Def_{q}(A)$.  The construction of the morphism is immediate: the scheme $Z \times_{H} \operatorname{SL}_{r}$
	admits a global contraction morphism given by projection onto the
	second factor, and the desired infinitesimal contraction is obtained
	by choosing a local inverse of $Z \times_{H} \operatorname{SL}_{r} \to \Quot(\calO^{\oplus r}_{X})$.

	We can use the contraction morphism to deduce the second claim, that $\Def_{q}(A) \to \Def_{I}(A)$ and $F(A) \to \Def_{I}(A)$
	have the same image.  Indeed, if $v \in \Def_{q}(A)$ maps to an element of
	$\mathfrak{sl}_{r}/\mathfrak{h}(A)$ represented by $g \in \mathfrak{sl}_{r}(A)$,
	then $g^{-1} \cdot v$ lies in $F(A)$.  Because both $v$ and $g^{-1} \cdot v$ map to the same element of $\Def_{I}(A)$,
	we have proven the claim.

	We also need to verify that $F( k[\epsilon] ) \to \Def_{I}(k[\epsilon])$ is injective.  This too can be proven
	using the contraction map, but  we must first relate the kernel of $F( k[\epsilon] ) \to \Def_{I}(k[\epsilon])$
	to the contraction.   Specifically, we claim the kernel equals the image of the orbit map.  It is immediate that
	the image is contained in the kernel, but the reverse inclusion requires
	more justification.  Thus, suppose
	$(q_1 \colon \calO_{X_1}^{\oplus r} \to I_1, j)$  is a 1st order deformation with the property that $(I_1, j)$ is the trivial deformation.
	Because $q$ induces an isomorphism on global sections, we can chose bases and use the identification $j$ to
	represent $q_1 \colon H^{0}(X, \calO_{X_1}^{\oplus r}) \to H^{0}(X, I_1)$ by a matrix $g$ that reduces to the identity modulo
	$\epsilon$.   The matrix $g$ may not lie in $\mathfrak{sl}_{r}(k[\epsilon])$, but if we set
	$\delta := \det(g)$, then the product $\delta \cdot g^{-1}$ does.  One may check that $\delta \cdot g^{-1}$
	maps to the deformation  represented by $(q_1 \colon \calO^{\oplus r}_{X_1} \twoheadrightarrow I_1,j)$, establishing the reverse
	inclusion.

	We now prove injectivity by showing directly that the image of the orbit map has trivial intersection with
	$F( k[\epsilon] )$.  Given $v$ in this intersection, the image in $\mathfrak{sl}_{r}/\mathfrak{h}(k[\epsilon])$ under the contraction
	morphism is zero because $v$ lies in $F(k[\epsilon])$.  But, as $v$ also lies in the image of the orbit map,
	the image under the composition
	$\Def_{q}(k[\epsilon]) \to \mathfrak{sl}_{r}/\mathfrak{h}(k[\epsilon]) \to \Def_{q}(k[\epsilon])$ of the contraction map with the orbit
	map is $v$. Thus, $v=0$, and the proof is complete.
\end{proof}

The following definition and lemma relate the stabilizer of a point of the Quot scheme to an automorphism group.
\begin{defi}\label{D:nathom}
	Let $I$ be a rank $1$, torsion-free sheaf on a nodal curve $X$; $q \colon \calO_{X}^{\oplus r} \twoheadrightarrow I$ a quotient map
	corresponding to a point $\tilde{x}$ belonging to some Quot scheme $\Quot(\calO_{X}^{\oplus r})$.  Assume
	\begin{itemize}
      	\item $I$ is generated by global sections;
	\item $q \colon H^{0}(X, \calO_{X}^{\oplus r}) \to H^{0}(X, I)$ is an isomorphism.
	\end{itemize}	
	\begin{enumerate}[(i)]	
		\item  If $\operatorname{Stab}(\tilde{x}) \subset \operatorname{SL}_{r}$ is the stabilizer under the natural action on $\Quot(\calO_{X}^{\oplus r})$,
			then the \textbf{natural homomorphism} 
			$$\operatorname{Stab}(\tilde{x}) \to \Aut(I)$$
			 is defined by sending $g  \in \operatorname{Stab}(\tilde{x})$ into the unique automorphism  $\alpha(g) \colon I \to I$ with the property that
			$\alpha(g) \circ q  = q \circ g^{-1}$ (which exists since $I$ is generated by the image of $H^{0}(X, \calO_{X}^{\oplus r})$).
         \end{enumerate}
         Assume additionally that $X$ is stable.  Let $p \colon X \hookrightarrow \bbP^N$ be a  $10$-canonical
	embedding with $(p,q)$ corresponding to the point  $\tilde{y}$ of the relative Quot scheme $\operatorname{Quot}( \calO_{X_g}^{\oplus r})$ (as in the proof of Fact \ref{F:univJac}).	
         \begin{enumerate}			
		\item[(ii)] If $\operatorname{Stab}(\tilde{y}) \subset \operatorname{SL}_{r} \times \operatorname{SL}_{N+1}$ is the stabilizer under the natural action on $\operatorname{Quot}( \calO_{X_g}^{\oplus r})$, then
			the \textbf{natural homomorphism} 
			$$\operatorname{Stab}(\tilde{y}) \to \Aut(X,I)$$
			 is defined by sending  $g = (g_1, g_2) \in \operatorname{Stab}(\tilde{y})$ into the unique element $\alpha(g) = (\alpha_1(g), \alpha_2(g)) \in \Aut(X,I)$ such that 			
			$p \circ \alpha_1(g) = g_2 \circ p$ and $\alpha_2(g) \circ \alpha_1(g)_{*}(q)  = q \circ g_{1}^{-1}$.
	\end{enumerate}
\end{defi}

\begin{lemma} \label{Lemma: StabToAut}
Same notation as in Definition \ref{D:nathom}. 
	\begin{enumerate}[(i)]	
	 	\item  \label{Lemma: StabToAut1}
The natural homomorphism $\operatorname{Stab}(\tilde{x}) \to \Aut(I)$ is injective with image equal to the subgroup 
			 $\Aut_{1}(I)\subset \operatorname{Aut}(I)$ consisting  of the automorphisms $\tau \in \Aut(I)$ with the property that
	the induced automorphism $H^{0}(X,I) \to H^{0}(X,I)$ has determinant $+1$.
		\item  \label{Lemma: StabToAut2}
The the natural homomorphism $\operatorname{Stab}(\tilde{y}) \to \Aut(X,I)$ is injective with image equal to the subgroup $\operatorname{Aut}_1(X,I) \subset \operatorname{Aut}(X,I)$ 
			consisting  of the automorphisms $(\sigma, \tau) \in \Aut(X,I)$ with the property that the induced automorphism
		         $H^{0}(X, I) \stackrel{{\rm can}}{\longrightarrow} H^{0}(X, \sigma_{*}(I)) \stackrel{\tau}{\longrightarrow} H^{0}(X,I)$ has determinant $+1$.
	\end{enumerate}
\end{lemma}
\begin{proof}
	As in the last proof, we only prove  the statement for $\operatorname{Stab}(\tilde{x})$ and
	leave the case of $\operatorname{Stab}(\tilde{y})$  to the interested reader.  Set $s_1, \dots, s_r \in H^{0}(X, I)$ equal
	to the image of the standard basis
	for $H^{0}(X, \calO_{X}^{\oplus r})$.  We first show injectivity.  Given $g \in \operatorname{Stab}(\tilde{x})$,
	write $(a_{i,j}) :=g^{-1}$.  Then $\alpha(g)$ satisfies
\begin{equation} \label{Eqn: ActionOfAlpha}
	\alpha( g )(s_i) = a_{i,1} s_1 + \dots + a_{i,r} s_r.
\end{equation}
	If $\alpha(g)$ is the identity, then we must have $\alpha(g)(s_i) = s_i$ for all $i$.  But the $s_i$'s form
	a basis, so this is only possible if $g = \operatorname{id}_{r}$, showing injectivity.  Similarly, given an
	$\alpha \in \Aut(I)$ that induces a determinant $+1$ automorphism of $H^{0}(X, I)$, define scalars
	$a_{i,j}$ as in Eqn.~\eqref{Eqn: ActionOfAlpha}.  Then $g := (a_{i,j})^{-1} \in \operatorname{SL}_{r}$
	is an element of $\Aut_1(I)$ with $\alpha(g) = \alpha$.  This completes the proof.
\end{proof}

The last lemma we need asserts that the formation of the relevant group quotients commutes with completion.
\begin{lemma} \label{Lemma: CompleteToInv}
	Let $Z$ be an affine algebraic scheme, $\tilde{x} \in Z$ a point, and $H$ an algebraic group acting on $Z$ that fixes $\tilde{x}$.  Assume $H$ is linearly reductive.
	Then the formation of $H$-invariants commutes with completion, i.e. if we call $x$ the image of $\tilde{x}$ in $Z/H$, then we have
	\begin{displaymath}
		\widehat{\calO}_{Z, \tilde{x}}^{H} \cong \widehat{\calO}_{Z/H, x}.
	\end{displaymath}
\end{lemma}
\begin{proof}

	This is an exercise in linear reductivity.  The quotient map induces
	a local homomorphism $\widehat{\calO}_{Z/H, x} \to \widehat{\calO}_{Z, \tilde{x}}$.  Because $\tilde{x}$
	is a fixed point, $H$ acts continuously on $\widehat{\calO}_{Z, \tilde{x}}$, and   passing to invariants, we may replace
	the target of this map with $\widehat{\calO}_{Z, \tilde{x}}^{H}$.  Our goal is to show that  the resulting map is an isomorphism.
	
	For injectivity, say $r \in \widehat{\calO}_{Z/H, x}$
	lies in the kernel.  By picking a sequence $\{ r_i \}_{i=1}^{\infty}$ in $\calO_{Z/H, x}$  converging to $r$ and
	studying the valuation of $r_i$, one can show that $r=0$.  Surjectivity requires more work.

	Given $r \in \widehat{\calO}_{Z, \tilde{x}}^{H}$, consider the reduction map $\widehat{\calO}_{Z, \tilde{x}}^{H} \to \widehat{\calO}_{Z, \tilde{x}}/\mathfrak{m}_{\tilde{x}}^{i+1}$.
	The element $r$ maps to an $H$-invariant element $\bar{r}$ in the target, which is canonically isomorphic to $\calO_{Z, \tilde{x}}/\mathfrak{m}_{\tilde{x}}^{i+1}$. Fixing
	an equivariant splitting of $\calO_{Z, \tilde{x}} \to \calO_{Z, \tilde{x}}/\mathfrak{m}_{\tilde{x}}^{i+1}$ (which exists by linear reductivity), we can
	lift $\bar{r}$ to an invariant element $r_i$ of $\calO_{Z, \tilde{x}}$.  The collection of all these elements defines a sequence
	$\{ r_i \}_{i=1}^{\infty}$ whose limit is $r$.  Furthermore, every term in the sequence lies in $\widehat{\calO}_{Z/H, x}$;
	thus the limit must lie in this ring as well.  This completes the proof.
\end{proof}

\begin{proof}[Proof of Theorem~\ref{Thm: MainLocStr}]
	The proof is an application of the Luna Slice Theorem,
	together with the previous lemmas.  As usual, we only give the proof for a 
	compactified Jacobian of a fixed nodal curve and leave the task of extending the argument to the universal
	compactified Jacobian to the interested reader (replacing Lemma \ref{Lemma: SliceToRing}\eqref{Lemma: SliceToRing1} with Lemma \ref{Lemma: SliceToRing}\eqref{Lemma: SliceToRing2}
        and  Lemma \ref{Lemma: StabToAut}\eqref{Lemma: StabToAut1} with Lemma  \ref{Lemma: StabToAut}\eqref{Lemma: StabToAut2} in the argument that follows).
        
        According to Corollary \ref{C:GITpres}, we can assume that $\bar{J}(X)\cong U/\!\!/\operatorname{SL}_r=U^{\rm ss}/\SL_r,$ where $U$ is the open subset of the Quot scheme  
        $\Quot(\calO_{X}^{\oplus r})$ defined in \emph{loc.~cit.}~ 
        
        Take now any  lift of $[I]\in \bar J(X)$  to a point $\tilde{x}\in U\subseteq \operatorname{Quot}(\calO_{X}^{\oplus r})$, corresponding to a quotient map
        $q \colon \calO^{\oplus r}_{X} \twoheadrightarrow I$, and observe that the orbit of $\widetilde{x}$ is closed in the semistable locus $U^{\rm ss}$ since $I$ is polystable.  
        Lemma \ref{Lemma: StabToAut}\eqref{Lemma: StabToAut1} identifies $\Stab(\widetilde{x})$ with the subgroup $\Aut_1(I)\subset \Aut(I)$ which is a (multiplicative) torus 
        (since $\Aut(I)$ is a torus by Lemma \ref{Lemma: aut-I} and any subgroup of a torus is a torus), hence linearly reductive.  
        Therefore, we can apply Luna Slice Theorem (see \S\ref{Subsec: GIT}) in order to get a slice $Z$ of $U$ at $\w{x}$.
         
	 Lemma~\ref{Lemma: SliceToRing}\eqref{Lemma: SliceToRing1}
	identifies  the ring $\widehat{\calO}_{Z,\tilde{x}}$ with the miniversal deformation ring $R_1$ of $\Def_I$.  Moreover, an exercise in unwinding the definitions shows that the 
	natural transformation $\pi:\Spf \widehat{\calO}_{Z,\tilde{x}}\to \Def_I$ is equivariant with respect to the natural homomorphism $\Stab(\tilde{x})\hookrightarrow \Aut(I)$ and the actions of 
	$\Stab(\tilde{x})$ on $\Spf \widehat{\calO}_{Z,\tilde{x}}$ and of $\Aut(I)$ on $\Def_I$. Therefore, Fact~\ref{Fact: Rim} implies that the natural identification $\widehat{\calO}_{Z,\tilde{x}}\cong R_1$ is 
	$\Stab(\tilde{x})\cong \Aut_1(I)$-equivariant.

	Now, applying Eqn. \eqref{E:sliceloc} together with  Lemma~\ref{Lemma: CompleteToInv}, we get 
	\begin{equation}\label{E:locring1}
	\widehat{\calO}_{\bar{J}(X), [I]} \cong \widehat{\calO}_{Z,\tilde{x}}^{\Stab(\tilde{x})}\cong R_1^{\Aut_1(I)}
	\end{equation}

	Now observe that the subgroup $\bbG_m \subset \Aut(I)$ of scalar automorphism acts trivially on
	$R_1$, as it follows by the explicit description of the action of $\Aut(I)$ on $R_1$ given in Theorem \ref{Thm: MainActionThm}. 
	Thus, the natural action of $\Aut(I)$ on $R_1$ factors through the quotient $\Aut(I)/\bbG_m$.  Because the natural map
	$\Aut_1(I) \to\Aut(I)/\bbG_m$ is surjective, we get that 
	\begin{equation}\label{E:locring2}
	R_1^{\Aut_1(I)}\cong R_1^{\Aut(I)}. 
	\end{equation}
	Combining \eqref{E:locring1} and \eqref{E:locring2}, we get the conclusion.
\end{proof}

In the introduction, we asked if Theorem~\ref{Thm: MainLocStr} remains valid when $X$
is allowed to have an automorphism of order $p$.  The condition on the automorphism group
was only used to apply the Luna Slice Theorem, which applies to actions of linearly reductive groups.
It is probably unreasonable to expect an analogue of the Slice Theorem to hold for actions of an arbitrary reductive group
(see \cite{martin}), but we only need an analogue for actions of $\Aut(X, I)$.  This group is an extension of the finite (reduced)
group $\Aut(X)$ by the multiplicative torus $\Aut(I)$, and it is known that the Slice Theorem holds for both the action of a
torus (it is linearly reductive) and for the action of a finite group (see e.g.~\cite[Prop.~2.2]{sga}).  Perhaps there is a Slice Theorem for actions of an extension of a torus by an arbitrary finite
group?

\vspace{0,2cm}

Finally, we can prove Theorem~\ref{Thm: MainThmB} from the introduction.

\begin{proof}[Proof of Theorem~\ref{Thm: MainThmB}]
Given Theorem~\ref{Thm: MainThmA}, this result follows from  \cite{local2}.  To establish Parts~(\ref{Thm: MainThmB1}) and (\ref{Thm: MainThmB2}) of Theorem~\ref{Thm:
MainThmB}, it is enough to fix a point $[I]\in \bar{J}(X)$ with $I$ polystable and prove the analogous statement about the completed local ring $
\widehat{\calO}_{\bar{J}(X), [I]}$.  For the remainder of the proof, we will work exclusively with $\widehat{\calO}_{\bar{J}(X), [I]}$. 

Theorem~\ref{Thm: MainThmA} identifies $\widehat{\calO}_{\bar{J}(X), [I]}$ with the $T_{\Gamma}$-invariant subring of
$R_I=\widehat{A}(\Gamma)[[W_1, \dots, W_{g_\Sigma}]]$, where $g_{\Sigma}:=g(X)-b_1(\Gamma)$. The ring $\widehat{A}(\Gamma)$ is the completion of the ring $A(\Gamma)$ defined in \cite[\S6]{local2} at the maximal ideal 
$\widetilde{\mathfrak{m}}:=(U_{\el}, U_{\er} \colon e \in \Sigma)$ and the action of $T_{\Gamma}$ on $\widehat{A}(\Gamma)$ is induced by the action on $A(\Gamma)$ defined in \emph{loc.~cit.} By \cite[Thm.~6.1]{local2}, the invariant subring of $A(\Gamma)$ is the cographic toric face ring $R(\Gamma)$ (from \cite[Def.~1.4]{local2}).  Thus, applying Lemma \ref{Lemma: CompleteToInv}, we get 
\begin{equation}\label{E:pres-loc}
	\widehat{\calO}_{\bar{J}(X), [I]} \cong \widehat{R}(\Gamma)[[W_1, \dots, W_{g_{\Sigma}}]]
\end{equation}
where $\widehat{R}(\Gamma)$ is the completion of $R(\Gamma)$ at the ideal $\mathfrak{m} :=\widetilde{\mathfrak{m}} \cap R(\Gamma)$ (which
appears in \cite[Prop.~4.6]{local2}).

We now prove Part~(\ref{Thm: MainThmB1}) of the theorem.  In \cite{local2}, it is proven that $R(\Gamma)$
is Gorenstein and has slc singularities (\cite[Thm.~5.7]{local2}), and these properties persist after  passing to a completion and adding power series variables.   

To establish Part \eqref{Thm: MainThmB2}, it is enough to show that the multiplicity $e_{\mathfrak{m}}(R(\Gamma))$ is equal to $1$
if and only if every element of $\Sigma$ corresponds to a separating edge of the dual graph $\Gamma_{X}$ of $X$.
The formula for $e_{\mathfrak{m}}(R(\Gamma))$ given in \cite[Thm. 5.7(vii)]{local2} shows that if $e_{\mathfrak{m}}(R(\Gamma))=1$ then $\Gamma\setminus E(\Gamma)_{\rm sep}$
has a unique totally cyclic orientation and this can only happen if $\Gamma\setminus E(\Gamma)_{\rm sep}$ is a disjoint union of points, i.e.~if $\Gamma$ is a tree.
As $\Gamma$ is obtained from $\Gamma_{X}$ by contracting the edges not in $\Sigma$, Part \eqref{Thm: MainThmB2} follows. 
\end{proof}

\section{Examples}\label{secexam}

In this section we present some examples to further elucidate the connections between the results in this paper and those of \cite{local2}.

\subsection{Integral curves}

Suppose that $X_0$ is an integral nodal curve of arithmetic genus $g$ and with $m$ nodes. From Definition \ref{D:phi-ss}, it follows that any compactified Jacobian of $X$  is equal to the (fine) moduli space 
$\bar{J}^d(X)$ of rank $1$, torsion-free sheaves on $X$ of degree $d$ (for some $d\in \bbZ$).

Consider now a point  $[I]\in \bar{J}^d(X)$ such that $I$ is not locally free at all the nodes of $X$ (note that $I$ is stable). Then Theorem \ref{Thm: MainThmA}\eqref{Thm: MainThmA1} gives that 
$$\widehat{\calO}_{\bar{J}^d(X),[I]}\cong  \wh{\bigotimes}_{i=1}^m  \frac{k[[X_i,Y_i]]}{X_iY_i}\wh{\bigotimes} k[[T_1,\ldots, T_{g-m}]]$$
We recover the well-know fact (see \cite[Prop. 2.7]{CMK}) that $\bar{J}^d(X_0)$ is isomorphic, formal (indeed \'etale) locally at $I$, to the product of $m$ nodes and a smooth factor.

\subsection{Two irreducible components} \label{sec2ir}
Let $X_n$ be a nodal curve consisting of two smooth irreducible components $C_1$ and $C_2$ of genera, respectively, $g_1$ and $g_2$, intersecting in $n\geq 2$ nodes, so that the arithmetic genus of $X_n$ is
$g=g_1+g_2+n-1$ (the case $n=1$ is easy: the curve $X_1$ is of compact type, hence all compactified Jacobians are smooth and isomorphic to the generalized Jacobian).
The dual graph $\Gamma_{X_n}$ of $X_n$ is depicted in Figure \ref{F:2comp} below together with an  orientation of it.

\begin{figure}[ht]\label{F:2comp}
\begin{equation*}
\xymatrix@=.5pc{
&&& \vdots &&& \\
 *{\bullet} \ar @{-}@/_1pc/[rrrrrr]|-{\SelectTips{cm}{}\object@{<}}^{e_{r+1}} \ar @{-} @/_3pc/[rrrrrr] |-{\SelectTips{cm}{}\object@{<}}_{e_n}
 \ar@{-} @/^.6pc/[rrrrrr]|-{\SelectTips{cm}{}\object@{>}}_{e_r}  \ar@{-}@/^2.5pc/[rrrrrr]|-{\SelectTips{cm}{}\object@{>}}^{e_1}^<{v_1}^>{v_2}  &&& &&&*{\bullet} \\
 &&& \vdots &&&
}
\end{equation*}
\caption{The orientation $\phi_r$ on $\Gamma_{X_n}$.}\label{Fig: 2vert}
\end{figure}
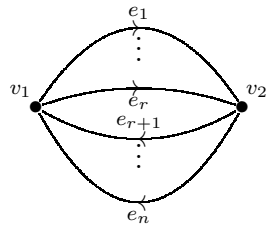
Let $\bar{J}(X_n)$ be a compactified Jacobian of $X_n$ and suppose that there exists a polystable sheaf $[I]\in \bar{J}(X_n)$ that fails to be locally free at the $n$ nodes of the curve.
Therefore, Eqn. \eqref{E:pres-loc} gives
$$\widehat{\calO}_{\bar{J}^d(X),[I]}\cong \widehat{R}(\Gamma_{X_n})[[W_1,\ldots, W_{g_1+g_2}]].$$

Using this presentation of the complete local ring and the results of \cite{local2}, we can prove the following properties:

$\bullet$ \'Etale locally at $[I]$, $\bar{J}(X_n)$ has $\displaystyle \sum_{r=1}^{n-1}\binom{n}{r}$ irreducible components, which are in bijection by \cite[Thm. 5.7(i)]{local2} with the  totally cyclic orientations
of $\Gamma_{n}$, all of which look like the orientation $\phi_r$ (for $1\leq r\leq n-1$) depicted in the figure above.

$\bullet$  The dimension of the Zariski tangent space $T_{[I]} \bar{J}(X_n)$ (i.e. the embedded dimension of $\bar{J}(X_n)$ at $[I]$) is equal to $g_1+g_2+2\binom{n}{2}$, as it follows from the fact (proved in \cite[Thm. 5.7(vi)]{local2}) that the embedded dimension of $R(\Gamma_{X_n})$ at the maximal ideal $\mathfrak{m}$ is equal to the number of oriented circuits of $\Gamma_{X_n}$, which is  $2\binom{n}{2}$.

$\bullet$ Finally, it can be proved using \cite[Thm. 5.7(vii)]{local2} that the multiplicity of $\bar{J}(X_n)$ at $[I]$ is 
$$\operatorname{mult}_{[I]}\bar J=\sum_{r=1}^{n-1}\binom{n}{r}\binom{n-2}{r-1}.$$

\bibliography{locstr}
\end{document}